\documentclass[11pt]{amsart}
\usepackage{amsmath,amssymb,latexsym,dsfont}

\usepackage[svgnames]{xcolor}
\usepackage{soul}
 \usepackage{tikz}
\usepackage{hyperref}
\usepackage{amscd}
\usepackage{amsfonts}
\usepackage{indentfirst}
\usepackage{verbatim}
\usepackage{amsmath}
\usepackage{amsthm}
\usepackage{enumerate}
\usepackage{graphicx}
\usepackage{subfig}
\usepackage{color}
\usepackage[OT1]{fontenc}
\usepackage[latin1]{inputenc}
\usepackage[english]{babel}
\usepackage{amssymb, mathabx}
\sethlcolor{LightBlue}
\makeatletter
\def\SOUL@hlpreamble{%
\setul{\dp\strutbox}{\dimexpr\ht\strutbox+\dp\strutbox\relax}
\let\SOUL@stcolor\SOUL@hlcolor
\SOUL@stpreamble
}
\makeatother

\newtheorem{theorem}{Theorem}[section]
\newtheorem{lemma}{Lemma}[section]

\newtheorem{proposition}{Proposition}[section]
\newtheorem{corollary}{Corollary}[section]
\newtheorem{remark}{Remark}[section]

\newtheorem{definition}{Definition}[section]
\numberwithin{equation}{section}
       \newcommand{\proofend}{\hfill $\Box$ }

      \newcommand{\R}{{\mathbb{R}}}

      \newcommand{\curl}{\operatorname{curl}}
      \newcommand{\dive}{\operatorname{div}}

      \newcommand{\loc}{\operatorname{loc}}
      
      \newcommand{\eps}{\varepsilon}
      \newcommand{\mR}{\mathbb{R}}
      \newcommand{\mC}{\mathbb{C}}

       \newcommand{\dist}{\mbox{dist}}

      \newcommand{\bE}{{{\bf E}}}
      \newcommand{\bH}{{{\bf H}}}
      
      \newcommand{\hbE}{{{\hat{\bf E}}}}
      \newcommand{\hbH}{{{\hat {\bf H}}}}

         \newcommand{\hE}{{\hat E}}
      \newcommand{\hH}{{\hat H}}

           \newcommand{\hmu}{{\hat \mu}}
      \newcommand{\heps}{{\hat \eps }}

           \newcommand{\mup}{\mu^+}
           \newcommand{\mun}{\mu^-}

           \newcommand{\mmu}{ {\bf m}}
           \newcommand{\eeps}{{\bf e}}

           \newcommand{\ep}{\eps^+}
           \newcommand{\en}{\eps^-}

        \newcommand{\hvarphi}{\hat \varphi}
        
         \newcommand{\hJ}{{\hat J }}
       
        \newcommand{\dE}{{\delta_E}}
        \newcommand{\dEn}{{\delta_{E_n}}}
        \newcommand{\dHn}{{\delta_{H_n}}}

        \newcommand{\dEd}{{\delta_{E_\delta}}}
        \newcommand{\dHd}{{\delta_{H_\delta}}}

          \newcommand{\bdkE}{\delta_{{\bf E}_k}}

         \newcommand{\dH}{{ \delta_H}}

        \newcommand{\dpE}{{\delta_{\varphi E}}}

        \newcommand{\dpkE}{{\delta_{\varphi_k E}}}

 \newcommand{\cE}{{{\mathcal  E}}}
            \newcommand{\cH}{{{\mathcal H}}}

     \newcommand{\supp}{\mbox{supp }}

      \makeatletter
      \def\@setcopyright{}
      \def\serieslogo@{}
      \makeatother

\begin{document}

  \author[H.-M. Nguyen]{Hoai-Minh Nguyen}
\author[S. Sil]{Swarnendu Sil}

\address[H.-M. Nguyen]{Department of Mathematics, EPFL SB CAMA, Station 8,  \newline\indent
	 CH-1015 Lausanne, Switzerland.}
\email{hoai-minh.nguyen@epfl.ch}

\address[S. Sil]{Forschungsinstitut f\"{u}r Mathematik, 
	ETH Zurich \newline\indent R\"{a}mistrasse 101, 8092 Zurich, Switzerland.}
\email{swarnendu.sil@fim.math.ethz.ch}

\title[Maxwell equation with anisotropic sign-changing coefficients]{Limiting absorption principle and well-posedness for the time-harmonic Maxwell equations with anisotropic sign-changing coefficients}

\begin{abstract}  We study the limiting absorption principle and the well-posedness of Maxwell equations with anisotropic sign-changing coefficients in the time-harmonic domain.  The starting point of the analysis is to obtain  Cauchy problems associated with two Maxwell systems  using a change of variables.   We then derive a priori estimates for these Cauchy problems using two different approaches. The Fourier approach involves the complementing conditions for the Cauchy problems associated with two elliptic equations,  which were  studied in a general setting by Agmon, Douglis, and Nirenberg. The  variational approach explores the variational structure of the Cauchy problems of the Maxwell equations.  As a result, we obtain general  conditions on the coefficients for which the  limiting absorption principle and the well-posedness hold. Moreover,  these {\it new} conditions are  of a local character  and easy to check.   Our work is motivated by and provides  general sufficient criteria for  the stability of electromagnetic fields in the context of  negative-index metamaterials.

\end{abstract}


   \date{\today}

   \maketitle

\medskip 
\noindent{\bf Key words}:  Maxwell equations, sign-changing coefficients, well-posedness,  limiting absorption principle, Cauchy problems, resonance,  negative-index metamaterials,

\medskip 

\noindent{\bf AMS subject classifications}: 35B34, 35B35, 35B40, 35J05, 78A25.

\tableofcontents



\section{Introduction}

Negative-index metamaterials  are artificial structures whose  refractive index has a negative value over some frequency range. Their existence was  postulated by Veselago  in 1964  \cite{Veselago} and  confirmed experimentally by Shelby, Smith, and Schultz in 2001 \cite{ShelbySmithSchultz}.   Negative-index metamaterial research  has been a very active topic of investigation  not only because of potentially interesting applications, but also because of challenges involved in understanding their peculiar properties due to the sign-changing coefficients in the equations modeling the phenomena.  

In  this paper, we study the stability of electromagnetic fields in the context of negative-index metamaterials from a mathematical point of view. More precisely, we  study the limiting absorption principle and the well-posedness of Maxwell equations with anisotropic sign-changing coefficients in the time-harmonic domain. 
Let $D$ be an open, bounded subset of $\mR^3$ of class $C^1$ and let $\ep, \mup$ be defined in $\mR^3 \setminus \bar D$,  and $\en, \mun$ be defined in $D$ such that 
$\ep, \mup$, $- \en$, and $- \mun$  are real,  symmetric,  {\it uniformly elliptic} matrix-valued functions.  Set, for $\delta \ge 0$,  
\begin{equation}\label{def-eDelta}
(\eps_\delta, \mu_\delta)  = \left\{\begin{array}{cl}
(\ep, \mup) & \mbox{ in } \mR^3 \setminus \bar D, \\[6pt]
(\en + i \delta I,  \mun + i \delta I)  & \mbox{ in } D. 
\end{array} \right. 
\end{equation}
Here and in what follows, $I$ denotes the $(3 \times 3)$ identity matrix and  $\bar D$ denotes the closure of $D$.  We also denote $B_R$ the open ball in $\mR^3$ centered at the origin and of radius $R>0$. 
We assume as usual  that for some $R_0 > 0$,   $D \subset B_{R_0}$, $(\ep, \mup) = (I, I)$ in $\mR^3 \setminus B_{R_0}$, and
\begin{equation}
\ep, \mup, \en, \mun \mbox{ are piecewise } C^1. 
\end{equation}

Given $\delta > 0$ and $J \in [L^2(\mR^3)]^3$ with compact support, let $(E_\delta, H_\delta) \in [H_{\loc}(\curl, \mR^3)]^2$ be  the unique radiating solution of the Maxwell equations 
\begin{equation}\label{Main-eq-delta}
\left\{\begin{array}{cl}
\nabla \times E_\delta = i \omega \mu_\delta H &  \mbox{ in } \mR^3, \\[6pt]
\nabla \times H_\delta = - i \omega \eps_\delta E + J &  \mbox{ in } \mR^3.  
\end{array} \right.
\end{equation}
Physically, $\eps_\delta$ and $\mu_\delta$ describe the permittivity and the permeability of the considered medium,   $\omega$ is the frequency,  $D$ is a plasmonic structure of negative-index metamaterials and 
$i \delta I$ describes its loss, and $J$ is the density of charge. 

The goals of this paper are to derive general conditions on $(\eps_0, \mu_0)$ for which $(E_\delta, H_\delta)$ are bounded and to characterize its limit.  The  corresponding scalar version of Maxwell equations in the frequency regime is the Helmholtz equation. The well-posedness of the Helmholtz equation with sign-changing coefficients has been studied intensively, with the initial investigations initiated by Costabel and 
Stephan \cite{CostabelErnst}. Later, the well-posedness and the Fredholms character were   investigated  using the integral method \cite{Ola} and  the  $T$-coercivity method \cite{AnneSophieChesnelCiarlet1, AnneSophieChesnelCiarlet1-1} (and the references therein). Recently, one of the authors \cite{Ng-WP} obtained general results on the limiting absorbtion principle and  the well-posedness for the Helmholtz equation. His approach is based on a priori estimates using Fourier and variational approaches for Cauchy problems associated with two elliptic equations derived naturally in this context.  
The use of the Cauchy problems in the context of the Helmholtz equations with sign-changing coefficients originally appeared in \cite{Ng-Complementary}.  The theory behind the Helmholtz equations  is now almost complete and is the state-of-the-art in this area.

There is very little in the mathematical literature about the stability of the (full) Maxwell equations with sign-changing coefficients.  The only known work in this direction that could help us 
is due to Bonnet-Ben Dhia, Chesnel, and Ciarlet \cite{AnneSophieChesnelCiarlet3} using the  $T$-coercivity method. They considered a bounded setting  with isotropic coefficients and obtained the well posedness outside a discrete set of $\omega$  under requirements on $\eps$ and $\mu$ of a non-local nature, which is generally difficult to check. The goal of this paper is to fill this gap by developing  the approach in \cite{Ng-WP}   for the Maxwell equations. The idea   is  first to obtain the Cauchy problems associated with two Maxwell systems via a change of variables  and then to derive a priori estimates for these Cauchy problems.   Two approaches  are proposed to 
obtain these estimates. The  Fourier approach involves the complementing conditions for Cauchy problems associated with two elliptic equations,  which were studied in more general setting by  Agmon, Douglis, and Nirenberg \cite{ADNII}. The variational approach explores the variational structure of the Cauchy problems. In comparison  with the acoustic setting \cite{Ng-WP}, 
new ideas are required to handle the complex structure of the Maxwell equations
and to be able to apply the theory on complementing conditions for elliptic systems. In particular, various forms of the Poincar\'e lemma and  the Helmholtz decomposition are used  with a suitable implementation of local charts to avoid imposing topological conditions.  
As a result, we obtain very general  conditions on the (anisotropic) coefficients to ensure the well-posedness and the limiting absorption principle. Moreover, these {\it new} conditions are local and easy to check. 
For example, the well-posedness and the limiting absorption principle  hold under the condition $i)$ $(\ep, -\en)$ and  $(\mup, -\mun)$ 
are smooth near the interface $\partial D$ and both satisfy the complementary condition on the interface (Theorem~\ref{thm1} and Corollay~\ref{corADN}), or  $ii)$ $\eps$ and $\mu$ are isotropic near the interface and $|\ep +  \en|$ and  $|\mup + \mun|$ are away from $0$ there (Theorem~\ref{thm2} and Corollary~\ref{corisotropic-2}), or $iii)$ $\ep,  \, -\en, \, \mup, \,  \mun$ are equal to $I$ near the interface and the interface is (strictly) convex (Corollary~\ref{corisotropic-3}).

 Our work thus provides general sufficient criteria for  the stability of electromagnetic fields in the context of negative-index metamaterials in the electromagnetic setting. This work also clarifies  the role of the complementary conditions considered in \cite{Ng-Superlensing-Maxwell} in ensuring the resonance associated with negative-index metamaterials  appeared in their various applications,  such as superlensing \cite{Ng-Superlensing-Maxwell} and cloaking \cite{Ng-CALR-Maxwell, Ng-Cloaking-Maxwell},  in the electromagnetic setting. Previous mathematical works on applications of negative-index metamaterials,  such as superlensing, 
cloaking using complementary media,  cloaking via anomalous localized resonance for a source or for an object,  can be found in  \cite{Ng-Superlensing}, \cite{Ng-Negative-Cloaking, MinhLoc2},  \cite{A-M1,  KangLY, KohnLu, MiltonNicorovici, Ng-CALR, Ng-CALR-F, Ng-CALR-Maxwell, NicoroviciMcPhedranMilton94},  and \cite{Ng-CALR-object}, respectively, and the references therein.

\section{Statement of the main results}

The starting point of our analysis  is to obtain Cauchy problems associated with two Maxwell systems using a change of variables.   
A central point of the analysis is then to derive a priori estimates for these Cauchy problems.
We first use  a  Fourier approach involving the theory of complementing conditions as developed by  Agmon, Douglis, and Nirenberg \cite{ADNII}. Given  a unit vector $e \in \mR^3$, denote  
$$
\mR^3_{e, +} = \Big\{ x \in \mR^3; \langle x, e \rangle > 0 \Big\} \quad \mbox{ and } \quad \mR^3_{e, 0} = \Big\{ x \in \mR^3; \langle x, e \rangle = 0 \Big\}. 
$$
Here and in what follows,  $\langle \cdot, \cdot \rangle$ denotes the standard scalar product in $\mR^3$ or $\mC^3$.  
The definition of the complementing condition for the Cauchy problem of two elliptic equations is as follows 
 
\begin{definition}[Agmon, Douglis, Nirenberg \cite{ADNII}]  Two constant,  positive, symmetric matrices $A_1$ and $A_2$ are said to satisfy the (Cauchy) complementing boundary condition with respect to  direction $e  \in \partial B_1$ if and only if for all $\xi \in \mR^3_{e, 0} \setminus \{0 \}$, the only solution $(u_1(x), u_2(x))$ of the form $\big(e^{i \langle y,  \xi \rangle  } v_1(t), e^{i \langle y,  \xi \rangle} v_2(t) \big)$ with $x  =  y + t e$ where  $t = \langle x, e \rangle $, of the  following system 
\begin{equation*}
\left\{ \begin{array}{c}
\dive(A_1 \nabla u_1) = \dive(A_2 \nabla u_2) = 0 \mbox{ in } \mR^3_{e, +}, \\[6pt]
u_1 = u_2 \mbox{ and } A_1 \nabla u_1 \cdot e = A_2 \nabla u_2 \cdot e \mbox{ on } \mR^3_{e, 0}, 
\end{array}\right.
\end{equation*}
that is bounded in $\mR^3_{e, +}$ is $(0, 0)$. 
\end{definition}

\begin{remark} \rm Agmon, Douglis, and Nirenberg \cite{ADNII} (see also \cite{ADNI})  considered the complementing conditions for a general elliptic system  and derived properties for them. The special case given in the above definition found a particular interest in the context of the Helmholtz equations with sign-changing coefficients (see \cite{Ng-WP}). 

\end{remark}

Denote
$$
\Gamma = \partial D, 
$$
and, for $\tau > 0$, set
\begin{equation*}
D_\tau = \Big\{x \in D;\;  \dist(x, \Gamma) < \tau \Big\} \quad \mbox{ and } \quad D_{-\tau} = \Big\{x \in \mR^3 \setminus \bar D; \; \dist(x, \Gamma)  < \tau \Big\}.
\end{equation*}

For the Fourier approach, we prove the following result: 

\begin{theorem}  \label{thm1}
Let $0 < \delta < 1$, $J \in [L^2(\mR^3)]^3$ with $\supp J \subset B_{R_0}$, and  
let $(E_\delta, H_\delta) \in [H_{\loc} (\curl, \mR^3)]^2$ be the unique radiating solution of \eqref{Main-eq-delta}. Assume that $\ep, \, \mup \in C^1(\bar D_{-\tau})$, $\en, \, \mun \in C^1(\bar D_{\tau})$ for some $\tau > 0$, and $(\ep,- \en)(x)$ {\rm and} $(\mup, - \mun)(x)$ satisfy the $($Cauchy$)$ {\rm complementing} conditions with respect to the unit normal vector $\nu(x)$ for all $x \in \Gamma$. Then 
\begin{equation}\label{T1-1}
\| (E_\delta, H_\delta) \|_{L^2(B_R)} \le C_R \| J\|_{L^2(\mR^3)} \quad \forall \, R > 0, 
\end{equation}
for some positive constant $C_R$ independent of $\delta$ and $J$.  Moreover,  $(E_\delta, H_\delta)$ converges to $(E_0, H_0)$ strongly in $[L^2_{\loc}(\mR^3)]^6$ as $\delta \to 0$, where  $(E_0, H_0) \in [H_{\loc}(\curl, \mR^3)]^2$ is the  unique radiating solution of \eqref{Main-eq-delta} with $\delta = 0$.  As a consequence, we have
\begin{equation}\label{T1-2}
\| (E_0, H_0) \|_{L^2(B_R)} \le C_R \| J\|_{L^2(\mR^3)} \quad \forall \, R > 0. 
\end{equation}
\end{theorem}

Let $\Omega$ be an open set of $\mR^3$. Given $(u_n) \subset L^2_{\loc}(\Omega)$ and $u \in L^2_{\loc}(\Omega)$, 
one says that $u_n$ converges to $u$ in $L^2_{\loc}(\Omega)$ if $u_n$ converges to $u$ in $L^2(K)$ for every compact subset $K$ of $\Omega$.  Recall that, for $\omega> 0$, a solution $(E, H) \in [H_{\loc}(\curl, \R^3\setminus B_R)]^2$, for some $R> 0$, of the Maxwell equations 
\[
\begin{cases}
\nabla \times E = i \omega H  &\text{ in } \mathbb{R}^3\setminus B_R,\\[6pt]
\nabla \times H = -i \omega E  &\text{ in } \mathbb{R}^3\setminus B_R,
\end{cases}
\]
is called radiating if it satisfies one of the (Silver-M\"{u}ller) radiation conditions
\begin{equation*}
H \times x - |x| E = O(1/|x|) \quad   \mbox{ or } \quad  E\times x + |x| H = O(1/|x|) \qquad \mbox{ as } |x| \to + \infty. 
\end{equation*}
Here and in what follows, for $\alpha \in \mR$, $O(|x|^\alpha)$ denotes a quantity whose norm is bounded by $C |x|^\alpha$ for some constant $C>0$. One also denotes  
\begin{equation*}
H(\curl, \Omega) = \Big\{u \in [L^2(\Omega)]^3; \nabla \times u \in [L^2(\Omega)]^3 \Big\}, 
\end{equation*}
\begin{equation*}
H_{\loc}(\curl, \Omega) = \Big\{u \in [L^2_{\loc}(\Omega)]^3; \nabla \times u \in [L^2_{\loc}(\Omega)]^3 \Big\}, 
\end{equation*}
\begin{equation*}
H (\dive, \Omega) = \Big\{u \in [L^2(\Omega)]^3;  \dive u \in L^2(\Omega) \Big\},  
\end{equation*}
and, on $\partial \Omega$,  $\nu$  the unit normal vector directed to the exterior of $\Omega$. 

\medskip

The proof of Theorem~\ref{thm1} is given in Section~\ref{sect-thm1}. The existence and uniqueness of $(E_0, H_0)$ are also established there. The analysis given in Section~\ref{sect-thm1} requires a number of tools and results, such as  Poincar\'e's lemma,  the Helmholtz decomposition for both electric and magnetic fields, the analysis on the complementing conditions,  duality arguments, and a  result on the trace estimate for $H(\dive, \Omega)$ (Lemma~\ref{lem-trace}), which is interesting in its own right. From the  analysis, we derive that  the complementing conditions for the Cauchy problem associated with two  Maxwell systems follows from the complementing conditions for the Cauchy problems associated with two elliptic equations corresponding to the pair of permittivity and permeability.  To the best of our knowledge, this insight is new, and its discovery makes the  required conditions on $\eps$ and $\mu$ easy to check (see Proposition~\ref{pro-C} below). The optimality of the complementing conditions is discussed in Proposition~\ref{pro-opt},  whose proof is given in the appendix.

\medskip 
To check the complementing condition, one can use its following algebraic characterization (see \cite[Proposition 1]{Ng-WP}). 

\begin{proposition}\label{pro-C}  Let $e \in \partial B_1$,  and let $A_1$ and $A_2$ be two constant, positive, symmetric matrices. Then  $A_1$ and $A_2$ satisfy the $($Cauchy$)$ complementing  condition with respect to $e$ if and only if 
\begin{equation*}
\langle A_2 e, e \rangle \langle A_2 \xi, \xi \rangle  - \langle A_2 e,  \xi \rangle^2 \neq  \langle A_1 e, e \rangle \langle A_1 \xi, \xi \rangle  - \langle A_1 e, \xi \rangle^2 \quad \forall \, \xi \in \mR^3_{e, 0} \setminus \{0 \}. 
\end{equation*}
In particular, if $A_2 > A_1$,  then $A_1$ and $A_2$ satisfy the complementing  condition with respect to $e$. 
\end{proposition}

In this paper, for two  $(3 \times 3)$ matrices $A$ and  $B$, the notation $A > B$ means that $\langle A x, x \rangle  > \langle B x, x \rangle$ for all $x \in \mR^3 \setminus \{0 \}$. A similar convention is used for $A \ge B$.  As a direct consequence of Theorem~\ref{thm1} and Proposition~\ref{pro-C}, one obtains

\begin{corollary}\label{corADN}
Let $0 < \delta < 1$, $J \in [L^2(\mR^3)]^3$ with $\supp J \subset B_{R_0}$,   and  
let $(E_\delta, H_\delta) \in [H_{\loc} (\curl, \mR^3)]^2$ be the unique radiating solution of \eqref{Main-eq-delta}. Assume that $\ep, \, \mup \in C^1(\bar D_{-\tau})$ and $\en, \,  \mun \in C^1(\bar D_{\tau})$ for some $\tau > 0$, and for each connected component of $\Gamma$, 
\begin{equation*}
\Big( \ep  \ge - \en + c I \quad  \mbox{ or } \quad - \en \ge  \ep +  c I \Big), 
\end{equation*}
and
\begin{equation*}
\Big( \mup   \ge - \mun + c I \quad  \mbox{ or } \quad - \mun \ge  \mup  + c I \Big),  
\end{equation*}
for some $c > 0$.  Then 
\begin{equation*}
\| (E_\delta, H_\delta) \|_{L^2(B_R)} \le C_R \| J\|_{L^2(\mR^3)} \quad \forall \, R > 0, 
\end{equation*}
for some positive constant $C_R$ independent of $\delta$ and $J$.  Moreover,  $(E_\delta, H_\delta)$ converges to $(E_0, H_0)$ strongly in $[L^2_{\loc}(\mR^3)]^6$ as $\delta \to 0$, where  $(E_0, H_0) \in [H_{\loc}(\curl, \mR^3)]^2$ is the  unique radiating solution of \eqref{Main-eq-delta} with $\delta = 0$.  As a consequence, we have
\begin{equation*} 
\| (E_0, H_0) \|_{L^2(B_R)} \le C_R \| J\|_{L^2(\mR^3)} \quad \forall \, R > 0. 
\end{equation*}
\end{corollary}

Theorem~\ref{thm1} (see also its consequence  Corollary \ref{corADN}) provides very general conditions for ensuring the well-posedness and the validity of the limiting absorption principle. One does not require that $\ep + \en$ and $\mup + \mun$ have the same sign on $\Gamma$.  This was previously out of reach.

\medskip 

Our next approach for addressing the Cauchy problems is a variational one. To this end, we first introduce some notations. 
\begin{definition} Let $\tau > 0$ and $U$ be a smooth, open subset of $\mR^3$ such that $\bar U \subset D$. A transformation ${\mathcal F}: D \setminus \bar U \to D_{-\tau}$ is said to be a reflection through $\Gamma$ if and only if ${\mathcal F}$ is a diffeomorphism and ${\mathcal F}(x) = x$ on $\Gamma $. 
\end{definition}

Here and in what follows, when we mention a diffeomorphism ${\mathcal F}:  \Omega \to \Omega'$ for two open subsets $\Omega$ and $\Omega'$ of $\mR^3$, we mean that ${\mathcal F}$ is a diffeomorphism, ${\mathcal F} \in C^1 (\bar \Omega)$, and ${\mathcal F}^{-1} \in C^1(\bar \Omega')$.  For a diffeomorphism ${\mathcal F}$ from $\Omega$ to $\Omega'$, for a matrix $A$ defined in $\Omega$, and for a vector field ${\mathcal E}$ defined in $\Omega$,  denote 
\begin{equation*}
{\mathcal F}_* A (x') = \frac{\nabla {\mathcal F}  (x)  A(x) \nabla  {\mathcal F} ^{T}(x)}{\mathcal{J}(x)}
\quad \mbox{ and } \quad {\mathcal F} *  {\mathcal E} (x') = \nabla {\mathcal F}^{-T} (x) {\mathcal E}(x), 
\end{equation*}
with $x ={\mathcal F}^{-1}(x')$ and $\mathcal{J}(x) = \det \nabla  {\mathcal F}(x)$. For 
a vector field $J$ defined in $\Omega$,  we set
\begin{equation*}
{\mathcal T}_* J (x') = \frac{J(x)}{\mathcal{J}(x)}. 
\end{equation*}
In what follows, $d_{\partial \Omega}$ denotes the distance function to $\partial \Omega$ for an open,  bounded subset $\Omega$ of $\mR^3$, i.e., 
$$
d_{\partial \Omega}(x) : = \min\Big\{ |y-x|; y \in \partial \Omega \Big\} \mbox{ for } x \in \mR^3. 
$$

\medskip 
The main result in this direction is 

\begin{theorem}\label{thm2}
Let $0 < \delta < 1$, $\tau > 0$, $J \in [L^2(\mR^3)]^3$ with $\supp J \subset B_{R_0}$, and  
let $(E_\delta, H_\delta) \in [H_{\loc} (\curl, \mR^3)]^2$ be the unique radiating solution of \eqref{Main-eq-delta}. Assume that there exist a smooth open subset $U$ of $\mR^d$ with $ \bar U \subset D $ and a reflection ${\mathcal F}$ through $\Gamma$ from $D \setminus \bar U $ onto $D_{-\tau}$ such that, with the notations 
$$
(\eps, \mu) = (\ep, \mup) \mbox{ in } D_{-\tau} \quad \mbox{ and } \quad (\heps, \hmu) =   ({\mathcal F}_* \en, {\mathcal F}_*\mun) \mbox{ in } D_{-\tau},  
$$
for each connected component $O$ of $D_{-\tau}$, 
\begin{equation*}
(\heps - \eps \ge c d_\Gamma^{\alpha_1} I  \mbox{ in } O \quad \mbox{ or } \quad  \eps - \heps \ge c d_\Gamma^{\alpha_1} I 
 \mbox{ in } O ) 
\end{equation*}
and
$$
(\hmu - \mu \ge c  d_\Gamma^{\alpha_2} I   \mbox{ in } O \quad \mbox{ or } \quad  \mu - \hmu \ge c d_\Gamma^{\alpha_2} I  \mbox{ in } O), 
$$
for some $0 \le \alpha_1, \alpha_2 < 2$ and for some positive constant $c$.  In the case $\alpha_1 + \alpha_2 >  0$ in $O$, additionally  assume  that 
$$
\supp J \cap  \big(O \cup {\mathcal F}^{-1}(O) \big) = \emptyset.
$$ 
Set, in $D_{-\tau}$, 
\begin{equation*}
\hE_\delta = {\mathcal F}*E_\delta \quad \mbox{ and } \quad \hH_\delta = {\mathcal F}*H_\delta  
\end{equation*}
and additionally assume  that $D$ is of class $C^2$.  Then, for all $R> 0$ and for all open set $V$  containing $\Gamma$, 
\begin{multline}\label{T2-1}
\int_{B_R \setminus V} |(E_\delta, H_\delta)|^2 + \int_{D_{-\tau}} |\langle (\eps - \heps) E_\delta, E_\delta \rangle | + |\langle (\mu - \hmu) H_\delta, H_\delta  \rangle | \\[6pt]+ \int_{D_{-\tau}} |(E_\delta - \hE_\delta, H_\delta - \hH_\delta)|^2  \le C \| J\|_{L^2(\mR^3)}^2, 
\end{multline}
for some positive constant $C = C_{R, V}$ independent of $\delta$ and $J$.  Moreover,  for any sequence $(\delta_n)$ converging to $0$, up to a subsequence,  $(E_{\delta_n}, H_{\delta_n})$ converges to $(E_0, H_0)$ strongly in $[L^2_{\loc}(\mR^3 \setminus \Gamma)]^6$ as $n \to + \infty$, where  $(E_0, H_0) \in [H_{\loc}(\curl, \mR^3 \setminus \Gamma)]^2$ satisfies 
\eqref{T2-1}  with $\delta = 0$ and  is  a radiating   solution of \eqref{Main-eq-delta} with $\delta = 0$. In the case $\alpha_1 = \alpha_2 = 0$ for all connected components of $D_{-\tau}$, the convergence holds in $[L^2_{\loc}(\mR^3)]^6$ for $\delta \to 0$,  and the limit $(E_0, H_0) \in [H_{\loc}(\curl, \mR^3)]^2$ is unique. 
\end{theorem}

The meaning of the radiating solution of $(E_0, H_0)$ in Theorem~\ref{thm2} is understood as follows. A pair $(E_0, H_0) \in [H_{\loc}(\curl, \mR^3 \setminus \Gamma)]^2$ satisfies 
\eqref{T2-1}  with $\delta = 0$ and  is  called a radiating   solution of \eqref{Main-eq-delta} with $\delta = 0$ if $(E_{0}, H_{0})$ is a {\it radiating} solution of 
\begin{equation*}
\left\{\begin{array}{cl}
\nabla \times E_0 = i \omega \mu_0 H &  \mbox{ in } \mR^3 \setminus \Gamma, \\[6pt]
\nabla \times H_0 = - i \omega \eps_0 E + J &  \mbox{ in } \mR^3 \setminus \Gamma, 
\end{array} \right.
\end{equation*}
and it satisfies 
\begin{equation}\label{def-radiating2}
(E_0 - \hE_0) \times \nu = 0  = (H_0 - \hH_0) \times \nu \mbox{ on } \Gamma. 
\end{equation}
Note that \eqref{def-radiating2} makes sense since, if $\alpha_1 + \alpha_2 > 0$ in $O$,  
$$
\nabla (E_0 - \hE_0) = i \omega \mu H_0  - i \omega \hmu \hH_0 = i \omega \mu(H_0 - \hH_0) + i \omega (\mu - \hmu) \hH_0 \mbox{ in } O
$$
and
$$
\nabla (H_0 - \hH_0) = i \omega \eps E_0  - i \omega \heps \hE_0 = i \omega \eps(E_0 - \hE_0) + i \omega (\eps - \heps) \hE_0 \mbox{ in } O.  
$$
One can then check,  in both cases  $\alpha_1 + \alpha_2 > 0$ in $O$ or $\alpha_1 + \alpha_2 = 0$ in $O$, that  
\begin{equation*}
E_0 - \hE_0 \in H(\curl, D_{-\tau}) \quad \mbox{ and } \quad H_0 - \hH_0 \in H(\curl, D_{-\tau}) 
\end{equation*}
using \eqref{T2-1}  with $\delta = 0$.

The proof of Theorem \ref{thm2} is given in Section \ref{sect-thm2}. The variational structure is explored in Lemma~\ref{lem-M2}.  Important ingredients in the proof are  two forms of Poincar\'e's lemmas: one  with a gain in the integrability (Lemmas~\ref{lem-extension-2} and \ref{lem-extension-3}) and one with a gain of regularity (Lemma~\ref{lem-curl-2}). We also suitably  use local charts to avoid imposing topological conditions.  Nevertheless, a more global approach in comparison with the proof of Theorem~\ref{thm1} is required in order to remove undesirable terms in the process of using local charts (see Remark~\ref{rem-proof-thm2}). Some applications of Theorem~\ref{thm2} are given in Section~\ref{sect-application}. 

Let  $B_r(x)$ denote the open ball in $\mR^3$ centered at $x$ and of radius $r$ for $x \in \mR^3$ and for $r > 0$. It is shown in \cite[Proposition 2.2]{Ng-CALR-Maxwell} that if $\heps = \eps$ and $\hmu = \mu$ in $B_r(x_0) \cap D_{-\tau}$ for some $x_0 \in \Gamma$ and $r >0$, then resonance may occur. These requirements on  $\eps, \mu$, $\heps$, and $\hmu$ are related to the complementary property of media \cite[Definition 1]{Ng-Superlensing-Maxwell}. This property  plays a role in the construction of lensing and cloaking devices using negative-index materials \cite{Ng-Superlensing-Maxwell, Ng-Cloaking-Maxwell} for which a localized resonance can take place,  an interesting phenomenon of negative-index metamaterials in which the solutions blow up in some regions and remain bounded in others as the loss goes to 0.

The rest of the paper is organized as follows. The proof of Theorems~\ref{thm1}  and \ref{thm2} are given in Sections~\ref{sect-thm1}   and \ref{sect-thm2}, repectively. In the appendix, we present the proof of Proposition~\ref{pro-opt} on the optimality of complementing conditions for the Cauchy problem associated with two Maxwell systems.  

\section{Fourier approach for the Cauchy problems} \label{sect-thm1}

This section is devoted to the proof of Theorem~\ref{thm1}. We first establish several lemmas
in Section~\ref{sect-thm1-1} which are used in the proof of Theorem~\ref{thm1}. The proof of Theorem~\ref{thm1} is given in Section~\ref{sect-thm1-2}. The optimality of the complementing conditions used in Theorem~\ref{thm1} 
is discussed in Proposition~\ref{pro-opt} whose proof is given in the appendix. 
\subsection{Preliminaries} \label{sect-thm1-1}

We begin this section by recall the following known form of Poincar\'e's lemma, see e.g. \cite[Theorem 3.4]{Girault}.

\begin{lemma} \label{lem-curl} Let  $\Omega$ be a  simply connected,  bounded, open subset of $\mR^3$ of class $C^1$. There exists a linear, continuous transformation $T: H(\curl, \Omega) \to [H^1(\Omega)]^3$ such that 
\begin{equation*}
\nabla \times T(u)  = \nabla \times  u \mbox{ in } \Omega. 
\end{equation*}
\end{lemma}

We next establish  the key ingredient of the proof of Theorem~\ref{thm1}. 

\begin{lemma}\label{lem-compact-1-*}  Let  $\Omega$ be a  simply connected,  bounded, open subset of $\mR^3$ of class $C^1$.  Let $\eps, \, \heps, \, \mu,  \,\hmu$ be real,  symmetric,   uniformly elliptic matrix-valued functions
defined  in $\Omega$ and  of class $C^1$.  Let $J_{e}, J_{m},  \hJ_{e}, \hJ_{m}  \in [L^2(\Omega)]^3$  and let $(E, H), (\hE, \hH) \in [H(\curl, \Omega)]^2$ be such that 
\begin{equation*}
\left\{ \begin{array}{cl}
\nabla \times E = i \omega \mu H + J_{e} & \mbox{ in } \Omega, \\[6pt]
\nabla \times H = - i \omega \eps E + J_{m}  & \mbox{ in } \Omega, 
\end{array}\right. \quad \quad 
\left\{ \begin{array}{cl}
\nabla \times \hE = i \omega \hmu \hH + \hJ_{e} & \mbox{ in } \Omega, \\[6pt]
\nabla \times \hH = - i \omega \heps \hE + \hJ_{m} & \mbox{ in } \Omega, 
\end{array}\right. 
\end{equation*}
and 
\begin{equation*}
E \times \nu = \hE \times \nu  \quad \mbox{ and } \quad H \times \nu = \hH \times \nu  \mbox{ on } \partial \Omega. 
\end{equation*}
Set 
\begin{equation*}
\bE = T(E), \quad \hbE = T(\hE), \quad \bH = T(H), \quad \hbH = T(\hH) \quad \mbox{ in } \Omega,    
\end{equation*}
where $T$ is the operator given in Lemma~\ref{lem-curl}.  We have 
\begin{enumerate}
\item[1)] If   $(\eps, \heps)$ satisfies  the complementing condition with respect to  $\nu(x)$  for all $x \in \partial \Omega$, then 
\begin{multline}\label{lem-compact-1-E-2}
\| (E - \bE, \hE - \hbE) \|_{L^2(\Omega)} \\[6pt]
 \le C \Big(  \|(\bE, \hbE, J_m, \hJ_m)\|_{L^2(\Omega)} +  \|\bE - \hbE \|_{H^{-1/2}(\partial \Omega)}+ \| (E - \bE, \hE - \hbE) \|_{[H^{1}(\Omega)]^*}  \Big). 
\end{multline}

\item[2)] If   $(\mu, \hmu)$ satisfies  the complementing condition with respect to  $\nu(x)$  for all $x \in \partial \Omega$,  then 
\begin{multline}\label{lem-compact-1-H-2}
\| (H - \bH, \hH - \hbH) \|_{L^2( \Omega)} \\[6pt]
\le C \Big( \|(\bH, \hbH, J_e, \hJ_e)\|_{L^2(\Omega)} +  \|\bH - \hbH \|_{H^{-1/2}(\partial \Omega)} + \| (H - \bH, \hH - \hbH) \|_{[H^{1}(\Omega)]^*} \Big). 
\end{multline}
\end{enumerate}
Here $C$ denotes a positive constant depending only on $\eps$, $\mu$, and $\Omega$. 
\end{lemma}

Here and in what follows,  $[H^{1}(\Omega)]^*$ denotes the dual space of $H^1(\Omega)$.

\begin{proof} We only  prove  assertion $1)$.  Assertion $2)$ can be obtained similarly and its proof is omitted.  From the properties of $T$, we have 
$$
\nabla \times (E - \bE) = \nabla \times (\hE - \hbE) = 0 \quad  \mbox{ in } \Omega. 
$$
Since $\Omega$ is simply connected, there exists $\varphi, \, \hvarphi \in H^1(\Omega)$ such that 
\begin{equation}\label{lem-compact-choice-2}
E - \bE = \nabla \varphi,  \quad \hE - \hbE = \nabla \hvarphi \quad  \mbox{ in } \Omega, \quad \mbox{ and } \quad  
\int_{ \Omega} \varphi = \int_{ \Omega} \hvarphi = 0. 
\end{equation}
We have, by the second equation of the system of $(E, H)$,  
\begin{equation}\label{lem-compact-p2-a}
\dive (\eps \nabla \varphi )  =  \dive (\eps E)  - \dive (\eps \bE)  = 
- \dive (\eps \bE) - \frac{i}{\omega} \dive (J_{m})  \mbox{ in } \Omega, 
\end{equation}
by the second equation of the system of $(\hE, \hH)$,
\begin{equation}\label{lem-compact-p2-b}
\dive (\heps \nabla \hvarphi)  = \dive (\heps \hE) -  \dive (\heps \hbE)= -  \dive (\heps \hbE)- \frac{i}{\omega} \dive (\hJ_{m}) \mbox{ in } \Omega, 
\end{equation}
\begin{equation}\label{lem-compact-p2-1}
(\nabla \varphi - \nabla \hvarphi) \times \nu = (E - \bE) \times \nu -  (\hE - \hbE) \times \nu = (\hbE - \bE) \times \nu \mbox{ on } \partial \Omega,
\end{equation}
and, on $\partial \Omega$,  
\begin{multline}\label{lem-compact-p3}
\Big(\eps \nabla \varphi   + \eps \bE +  \frac{i}{\omega} J_m \Big)  \cdot \nu - \Big( \heps \nabla \hvarphi   + \heps \hbE + \frac{i}{\omega} \hJ_m \Big) \cdot \nu  \\[6pt]
= \Big( \eps E  + \frac{i}{\omega} J_m  \Big) \cdot \nu - \Big(  \heps \hE  +  \frac{i}{\omega} \hJ_m  \Big) \cdot \nu \\[6pt] 
 = - \frac{1}{i \omega} \nabla \times (H - \hH) \cdot \nu  =  \frac{1}{i \omega} \dive_\Gamma [(H - \hH) \times \nu] =0. 
\end{multline}
Since $(\eps, \heps)$ satisfies the complementing condition, applying  \cite[Lemma 18]{NguyenNguyen_transmissioneigenvalue} (with $\lambda = 1$, $\Sigma_1 = 1$, $\Sigma_2 > 0$ sufficiently large)  and using  the standard arguments of  freezing coefficients, we derive from \eqref{lem-compact-p2-a}, \eqref{lem-compact-p2-b}, 
\eqref{lem-compact-p2-1},  and \eqref{lem-compact-p3}  that  
\begin{equation}\label{lem-compact-coucou}
\| (\nabla \varphi, \nabla  \hvarphi) \|_{L^2( \Omega)} \le C \Big(  \| (\bE, \hbE, J_m, \hJ_m) \|_{L^2(\Omega)} + \|\bE - \hbE \|_{H^{-1/2}(\partial \Omega)}+ \| (\varphi, \hvarphi) \|_{L^2(\Omega)} \Big).  
\end{equation} 

We claim that,  for $u \in H^1(\Omega)$ with $\int_{\Omega} u = 0$,  
\begin{equation*}
\| u \|_{L^2(\Omega)} \le C \| \nabla u \|_{[H^1(\Omega)]^*}. 
\end{equation*}
In fact, fix $\xi \in L^2(\Omega)$ arbitrary and set 
$ \xi_\Omega : = \frac{1}{|\Omega|}\int_{\Omega} \xi $. 
Let $\varphi \in H^1(\Omega)$ be the unique solution of  $- \Delta \varphi  = \xi - \xi_{\Omega}$ in $\Omega$ and $\partial_\nu \varphi = 0$ on $\partial \Omega$ with $\int_{\Omega} \varphi =0$. We have 
$$
\| \nabla \varphi\|_{L^2(\Omega)} \le C \| \xi \|_{L^2(\Omega)}
$$
and
$$
\left| \int_{\Omega} u \xi \right| = \left| \int_{\Omega} u ( \xi - \xi_\Omega) \right|  = \left| \int_{\Omega} \nabla u \nabla \varphi \right| \le \| \nabla u \|_{L^2(\Omega)}\| \nabla \varphi \|_{L^2(\Omega)},
$$
and the claim follows. 

Applying the claim, we have 
\begin{equation*}
\| (\varphi, \hvarphi) \|_{L^2(\Omega)} \le C \| ( \nabla \varphi, \nabla \hvarphi) \|_{[H^{1}(\Omega)]^*},  
\end{equation*}
and assertion $1)$ follows from \eqref{lem-compact-coucou} and the choice of $\varphi$ and $\hvarphi$ in \eqref{lem-compact-choice-2}.  \end{proof}

\begin{remark} \rm In the proof of Lemma~\ref{lem-compact-1-*}, we used \cite[Lemma 18]{NguyenNguyen_transmissioneigenvalue}. This lemma is in the spirit of the results given in   \cite{ADNII} due to Agmon, Douglis, and Nirenberg where more regular data are used. 
\end{remark}

We next establish a compactness result used in Lemma~\ref{lem-compact-1} which is a key ingredient of the proof of Theorem~\ref{thm1}. 

\begin{lemma} \label{lem-compact-pre}   Let $d \ge 2$,  and let $O \subset \mR^d$ be open, bounded,  and of class $C^1$.  Let $(u_n)$ be a bounded sequence in $L^2(O)$. Then, up to a subsequence, $(u_n)$ converges in $[H^1(O)]^*$.  
\end{lemma}

\begin{proof} Without loss of generality, one may assume that $(u_n)$ converges weakly to $u$ in $L^2(O)$. It suffices to prove that, up to a subsequence,   $(v_n)$ converges to $0$ in $[H^1(O)]^*$ where $v_n :  = u_n - u$ in $O$. Let  $\varphi_n \in H^1_0(O)$ be the unique solution of $-\Delta \varphi_n = v_n$ in $O$. Since $(v_n)$
is bounded in $L^2(O)$, it follows that $(\varphi_n)$ is bounded in $H^1(O)$. Without loss of generality, one may assume that $(\varphi_n)$ converges in $L^2(O)$. 
Then 
\begin{equation}\label{lem-compact-pre-1}
\int_{O} |\nabla \varphi_n|^2 = \int_{O} v_n \varphi_n \to 0 \mbox{ as } n \to + \infty. 
\end{equation}
We have, for   $\xi \in H^1(O)$, 
\begin{equation*}
\int_{O} v_n \xi = -  \int_{O} \Delta \varphi_n \xi = \int_{O} \nabla \varphi_n \nabla \xi - \int_{\partial O} \partial_\nu \varphi_n \xi. 
\end{equation*}
It follows that 
\begin{equation}\label{lem-compact-pre-2}
\| v_n \|_{[H^1(O)]^*} \le C \Big(  \| \nabla \varphi_n \|_{L^2(O)} + \| \partial_\nu \varphi_n \|_{H^{-1/2}(\partial O)} \Big). 
\end{equation}
Here and in what follows in the proof, $C$ denotes a positive constant depending only on $O$. 
Applying Lemma~\ref{lem-trace} below, we have 
\begin{equation*}
\| \partial_\nu \varphi_n \|_{H^{-1/2}(\partial O)}^2 \le C \| \nabla \varphi_n \|_{L^2(O)} \Big( \| \nabla \varphi_n \|_{L^2(O)} +  \| v_n\|_{L^2(O)} \Big) \mathop{\to}^{\eqref{lem-compact-pre-1}} 0 \mbox{ as } n \to + \infty. 
\end{equation*}
The conclusion now follows from \eqref{lem-compact-pre-1} and \eqref{lem-compact-pre-2}. 
\end{proof}

In the proof of Lemma~\ref{lem-compact-pre}, we use the following result which is interesting in itself. 

\begin{lemma}\label{lem-trace} Let $d \ge 2$,  and let $O \subset \mR^d$ be open, bounded,  and of class $C^1$. Let 
$u \in H(\dive, O)$. We have 
\begin{equation}\label{lem-trace-statement}
\| u \cdot \nu\|_{H^{-1/2}(\partial O)}^2 \le C \| u \|_{L^2(O)} \Big(\| u  \|_{L^2(O)} + \|\dive u\|_{L^2(O)} \Big), 
\end{equation}
for some positive constant $C$ depending only on $O$. 
\end{lemma}

\begin{remark}  \rm Lemma~\ref{lem-trace} is in the spirit of \cite[Lemma A.1]{HJNg-2}. Estimate \eqref{lem-trace-statement} is stronger than the standard  one which asserts that 
$$
\| u \cdot \nu\|_{H^{-1/2}(\partial O)} \le C  \Big(\| u  \|_{L^2(O)} + \|\dive u\|_{L^2(O)} \Big). 
$$
\end{remark}

\begin{proof} Using a density argument, one may assume that $u$ is smooth. Using local charts, it suffices to prove that 
\begin{equation}\label{lem-trace-p0}
\| u \cdot \nu\|_{H^{-1/2}(\partial \mR^{d}_+)}^2 \le C \| u \|_{L^2(\mR^d_+)} \big(\| u \|_{L^2(\mR^d_+)}  + \|\dive u\|_{L^2(\mR^d_+)}  \big), 
\end{equation}
for $u \in C^1(\mR^d)$ with support in the unit ball, where $\mR^d_+ : = \mR^{d-1} \times (0, + \infty)$. 
Define, for $(\xi', x_d) \in \mR^{d-1} \times \mR$,  
$$
{\mathcal F} v (\xi', x_d): = \int_{\mR^{d-1}} v(x', x_d) e^{- i x' \cdot \xi'} \, d x' \mbox{ for } v \in C^1_{c}(\mR^d). 
$$
We have 
\begin{multline*}
\| u \cdot \nu\|_{H^{-1/2}(\partial \mR^{d}_+)}^2  = \int_{\mR^{d-1}} | {\mathcal F} u_d (\xi', 0)|^2 \big(1 + |\xi'|^2 \big)^{-1/2} \, d \xi' \\[6pt]
 \le \frac{1}{2} \int_0^\infty \int_{\mR^{d-1}} |\partial_{x_d}{\mathcal F} u_d (\xi', x_d)|  \,  |{\mathcal F} u_d (\xi', x_d)| \big(1 + |\xi'|^2 \big)^{-1/2} \, d \xi'  \, d x_d. 
 \end{multline*}
Using H\"older's inequality, we obtain 
\begin{equation}\label{lem-trace-p1}
\| u \cdot \nu\|_{H^{-1/2}(\partial \mR^{d}_+)}^2  \le  \frac{1}{2}  \left( \int_{\mR^{d}_+} |\partial_{x_d} {\mathcal F}  u_d (\xi', 0)|^2  \big(1 + |\xi'|^2\big)^{-1} \, d \xi' d x_d \right)^{1/2} \left( \int_{\mR^{d}_+} |{\mathcal F}u_d (\xi', 0)|^2 d \xi' d x_d \right)^{1/2}.
\end{equation}
Since 
$$
|\partial_{x_d} {\mathcal F}  u_d (\xi', x_d)|= \big|{\mathcal F}( \dive u) (\xi', x_d )- \sum_{j=1}^{d-1} i \xi'_j {\mathcal F} u_j(\xi', x_d ) \big| \le |{\mathcal F}( \dive u) (\xi', x_d )| +  \sum_{j=1}^{d-1} |\xi'_j {\mathcal F} u_j(\xi', x_d )|, 
$$
it follows from Parseval's theorem that 
\begin{equation}\label{lem-trace-p2}
\int_{\mR^{d}_+} |\partial_{x_d} {\mathcal F}  u_d (\xi', x_d )|^2  \big(1 + |\xi'|^2 \big)^{-1} \, d \xi' d x_d \le C \int_{\mR^{d}_+} \big( |\dive u (x)|^2 + |u (x)|^2\big) \, dx.  
\end{equation}
By Parseval's theorem, we also have
\begin{equation}\label{lem-trace-p3}
 \int_{\mR^{d}_+} |{\mathcal F}u_d (\xi', x_d)|^2 d \xi' d x_d  \le C \int_{\mR^{d}_+}  |u(x)|^2 \, d x. 
\end{equation}
Combining \eqref{lem-trace-p1}, \eqref{lem-trace-p2}, and \eqref{lem-trace-p3} yields \eqref{lem-trace-p0}. 
\end{proof}

As a consequence of Lemma~\ref{lem-compact-1-*}, one can derive the following compactness result for Maxwell's equations, which is the key ingredient  in the proof of Theorem~\ref{thm1}. 

\begin{lemma}\label{lem-compact-1}  Let $\Omega \subset \mR^3$ be open, bounded, and of class $C^1$ and let $\eps, \; \heps, \; \mu, \; \hmu$  be real,  symmetric,   uniformly elliptic matrix-valued functions
defined  in $\Omega$ and  of class $C^1$. 
Let $(J_{e, n}), \, (J_{m, n}), \, (\hJ_{e, n})$, and  $(\hJ_{m, n})  \subset  [L^2(\Omega)]^3$,  and  let $(E_n, H_n), \, (\hE_n, \hH_n) \in [H(\curl, \Omega)]^2$ be such that 
\begin{equation*}
\left\{ \begin{array}{cl}
\nabla \times E_n = i \omega \mu H_n + J_{e, n} & \mbox{ in } \Omega, \\[6pt]
\nabla \times H_n = - i \omega \eps E_n + J_{m, n}  & \mbox{ in } \Omega, 
\end{array}\right. \quad \quad 
\left\{ \begin{array}{cl}
\nabla \times \hE_n = i \omega \hmu \hH_n + \hJ_{e, n} & \mbox{ in } \Omega, \\[6pt]
\nabla \times \hH_n = - i \omega \heps \hE_n + \hJ_{m, n} & \mbox{ in } \Omega, 
\end{array}\right. 
\end{equation*}
and 
\begin{equation*}
E_n \times \nu = \hE_n \times \nu  \quad \mbox{ and } \quad H_n \times \nu = \hH_n \times \nu  \mbox{ on } \partial \Omega. 
\end{equation*}
Let $S$ be a connected component of $\partial \Omega$ and assume that 
\begin{equation*}
(\eps, \heps)(x), (\mu, \hmu)(x) \mbox{ satisfy the complementing condition with respect to  $\nu(x)$  for all $x \in S$}, 
\end{equation*}
\begin{equation*}
(E_n, H_n, \hE_n, \hH_n) \mbox{ are bounded in } [L^2(\Omega)]^{12}, 
\end{equation*}
and
\begin{equation*}
(J_{e, n}, J_{m, n}, \hJ_{e, n}, \hJ_{m, n}) \mbox{ converges in } [L^2(\Omega)]^{12} \mbox{ as } n \to + \infty. 
\end{equation*}
There exists a neighborhood $U$ of $S$ such that,  up to a subsequence, $(E_n, H_n, \hE_n, \hH_n)$ converges in $[L^2(U \cap \Omega)]^{12}$. 
\end{lemma}

\begin{proof} Let $\ell \ge 1$, $x_k \in S$, $r_k > 0$, and $\varphi_k \in C^1(\mR^3)$ for $1 \le k \le \ell$ be such that  $B_{r_k}(x_k) \cap \Omega$ is  simply connected with connected boundary, $\supp \varphi_k \Subset B_{r_k/2}(x_k)$, and $\sum_{k =1}^\ell \varphi = 1$ in a neighborhood $U$ of $S$.  For $1 \le k \le \ell$, set, $O_k = B_{r_k}(x_k) \cap \Omega$,  and,  in $O_k$,
define  
\begin{equation*}
\bE_{k, n} = T(\varphi_k E_{n}), \quad \hbE_{k, n} = T(\varphi_k \hE_n), \quad \bH_{k, n} = T(\varphi_kH_n), \quad \hbH_{k, n} = T(\varphi_k \hH_n),    
\end{equation*}
where $T$ is the operator given in Lemma~\ref{lem-curl} for the set $O_k$. By Lemma~\ref{lem-curl}, we have 
\begin{equation*}
\| (\bE_{k, n}, \hbE_{k, n}, \bH_{k, n}, \hbH_{k, n}) \|_{H^1(O_k)}  \\[6pt]
\le C \| (E_n, \hE_n, H_n, \hH_n, J_{e, n}, \hJ_{e, n}, J_{m, n}, \hJ_{m, n}) \|_{L^2(O_k)}. 
\end{equation*}
 Without loss of generality, one may assume that, as $n \to + \infty$,  
$$
(\bE_{k, n}, \hbE_{k, n}, \bH_{k, n}, \hbH_{k, n}) \mbox{ converges in } [L^2(O_k)]^{12}, 
$$
$$
(\bE_{k, n}- \hbE_{k, n}, \bH_{k, n}- \hbH_{k, n}) \mbox{ converges in } [H^{-1/2}(\partial O_k)]^{6}, 
$$
and, by Lemma~\ref{lem-compact-pre}, 
$$
(\varphi_k E_n - \bE_{k, n}, \varphi_k\hE_{k, n} - \hbE_{k, n}, \varphi_kH_n - \bH_{k, n}, \varphi_k \hH_n - \hbH_{k, n}) \mbox{ converges in } 
\big[[H^1(O_k)]^*\big]^{12}. 
$$
Applying  Lemma~\ref{lem-compact-1-*} \footnote{The complementing conditions holds on $\partial O_k \cap S$ a priori. However, since $\varphi_k = 0$ in  $B_{r_k}(x_k) \setminus B_{r_k/2}(x_k)$, one can modify $\eps, \, \mu, \, \heps, \, \hmu$ in $B_{r_k}(x_k) \setminus B_{2r_k/3}(x_k)$ so that the complementing conditions hold on the whole boundary and the systems are unchanged, e.g. $\eps > \heps$ and $\mu > \hmu$  in  $B_{r_k}(x_k) \setminus B_{r_k/2}(x_k)$ (see Proposition~\ref{pro-C}).}, we have 
\begin{multline*}
 \| \big( (\varphi_k E_n - \bE_{k, n}) - (\varphi_k E_l - \bE_{k, l}),  (\varphi_k \hE_n - \hbE_{k, n}) - (\varphi_k \hE_l - \hbE_l) \big) \|_{L^2(O_k)} \\[6pt]
 \le C \| \big( (\varphi_k E_n - \bE_{k, n}) - (\varphi_k E_l - \bE_{k, l}) ,  (\varphi_k \hE_n - \hbE_{k, n}) - (\varphi_k \hE_l - \hbE_{k, l})  \big) \|_{[H^1(O_k)]^*}\\[6pt]
 + C \|(\bE_{k, n} - \bE_{k, l}, \hbE_{k, n} - \hbE_{k, l}, J_{m, n} - J_{m, l} , \hJ_{m, n} - \hJ_{m, l})\|_{L^2(O_k)} \\[6pt]
 + C \| (\bE_{k, n} - \hbE_{k, n}) - (\bE_{l, n} - \hbE_{l, n}) \|_{H^{-1/2}(\partial O_k)} 
  \to 0 \mbox{ as  $l, n \to + \infty$,} 
\end{multline*}  and
\begin{multline*}
 \| \big( (\varphi_k H_n - \bH_{k, n}) - (\varphi_k H_l - \bH_{k, l}),  (\varphi_k \hH_n - \hbH_{k, n}) - (\varphi_k \hH_l - \hbH_l) \big) \|_{L^2(O_k)} \\[6pt]
\le  C \| \big( (\varphi_k H_n - \bH_{k, n}) - (\varphi_k H_l - \bH_{k, l}) ,  (\varphi_k \hH_n - \hbH_{k, n}) - (\varphi_k \hH_l - \hbH_{k, l})  \big) \|_{[H^1(O_k)]^*}\\[6pt]
+C \|(\bH_{k, n} - \bH_{k, l}, \hbH_{k, n} - \hbH_{k, l}, J_{e, n} - J_{e, l} , \hJ_{e, n} - \hJ_{e, l})\|_{L^2(O_k)} \\[6pt]
+ C  \| (\bH_{k, n} - \hbH_{k, n}) - (\bH_{l, n} - \hbH_{l, n}) \|_{H^{-1/2}(\partial O_k)}  \to 0 \mbox{ as  $l, n \to + \infty$.}  
\end{multline*}
Hence $(\varphi_k E_{k, n}, \varphi_k\hE_{k, n}, \varphi_k H_{k, n}, \varphi_k\hH_{k, n})$ is a Cauchy sequence in $[L^2(O_k)]^{12}$ and  therefore convergence in   $[L^2(O_k)]^{12}$. 
The conclusion then follows  since  $\sum_{k=1}^\ell \varphi_k = 1$ in $U$. 
\end{proof}

We next recall the following well-posedness result on \eqref{Main-eq-delta} (see \cite[Lemma 6]{Ng-Superlensing-Maxwell}). 

\begin{lemma} \label{lem-stability} Let $0 < \delta < 1$, $f, g \in [L^2(\mR^3)]^3$, and $(\eps_\delta, \mu_\delta)$ be defined in \eqref{def-eDelta}. Assume that
 $\supp f, \, \supp g  \subset B_{R_0}$. 
There exists a unique radiating  solution $({\cE_\delta, \cH_\delta}) \in [H_{\loc}(\curl, \mR^3)]^2$ to 
\begin{equation*}
\left\{ \begin{array}{lll}
\nabla \times \cE_\delta &= i \omega \mu_\delta \cH_\delta + f & \mbox{ in } \mR^3,\\[6pt]
\nabla \times \cH_\delta & = - i \omega  \eps_\delta \cE_\delta + g & \mbox{ in }  \mR^3. 
\end{array} \right.
\end{equation*}
Moreover,  we have
\begin{equation*}
\|(\cE_\delta, \cH_\delta) \|_{H(\curl, B_{R})} \le \frac{C_{R}}{\delta} \|( f, g) \|_{L^2}, 
\end{equation*}
where $C_R$ denotes a positive constant depending on $R$, $R_0$, $\eps$, and $\mu$,  but independent of $f$, $g$, and $\delta$. 
\end{lemma}

We now deal with  the uniqueness of \eqref{Main-eq-delta} for $\delta = 0$. 

\begin{lemma}\label{lem-unique-1} Let $J = 0$ and  $(E_0, H_0) \in [H_{\loc}(\curl, \mR^3)]^2$ is a radiating solution of  \eqref{Main-eq-delta} with $\delta = 0$. Then 
\begin{equation*}
E_0 = H_0 = 0 \mbox{ in } \mR^3.
\end{equation*}
\end{lemma}

\begin{proof} The proof is quite standard as in the  usual  case.  We present it here for the completeness.  Multiplying the equation of $E_0$ by $\nabla \times \bar E_0$ ($\bar E_0$ denotes the conjugate of $E_0$), we have 
\begin{equation*}
\int_{B_{R_0}} \langle \mu^{-1} \nabla \times E_0, \nabla \times E_0 \rangle = \int_{B_{R_0}} \langle i \omega H_0, \nabla \times E_0 \rangle.  
\end{equation*}
Integrating by parts the RHS and using the equation of  $H_0$, we obtain 
\begin{equation*}
\int_{B_{R_0}} \langle \mu^{-1} \nabla \times E_0, \nabla \times E_0 \rangle = \int_{B_{R_0}} \omega^2 |E_0|^2 + \int_{\partial B_{R_0}} i \omega (H_0 \times \nu) \cdot \bar E_0.   
\end{equation*}
This implies 
\begin{equation*}
\Re  \int_{\partial B_{R_0}} (H_0 \times \nu) \cdot \bar E_0 = 0. 
\end{equation*}
It follows, see, e.g., \cite[Theorem 6.10]{ColtonKressInverse}, that $E_0 = H_0 = 0$ in $\mR^3 \setminus B_{R_0}$. Therefore, $E_0 = H_0 = 0$ in $\mR^3$ by the unique continuation principle, see \cite{BallCapdeboscq, Tu} (see also \cite{Leis, Protter60}).   
\end{proof}

\subsection{Proof of Theorem~\ref{thm1}} \label{sect-thm1-2}

We first prove \eqref{T1-1} with $r_0 := R_0 + 1$ by contradiction. Assume that there exist a sequence $(\delta_n) \subset (0, 1)$ and a sequence $(J_n) \subset [L^2(\mR^3)]^3$ with $\supp J_n \subset B_{R_0}$ such that 
\begin{equation}\label{thm1-contradiction}
\| (E_n, H_n) \|_{L^2(B_{r_0})}  = 1 \quad \mbox{ and }  \quad \lim_{n \to + \infty} \| J_n\|_{L^2(\mR^3)} =0.  
\end{equation}
Here $(E_n, H_n)$ is the corresponding solution of \eqref{Main-eq-delta} with $\delta = \delta_n$ and $J = J_n$.  By Lemma~\ref{lem-stability}, without loss of generality, one may assume that $\delta_n \to 0$.   Since 
\begin{equation*}
\left\{ \begin{array}{cl}
\nabla \times E_n = i \omega H_n  & \mbox{ in } \mR^3 \setminus B_{R_0}, \\[6pt]
\nabla \times H_n = - i \omega E_n  & \mbox{ in } \mR^3 \setminus B_{R_0}, 
\end{array}\right. 
\end{equation*}
we have, for $R > R_0 + 1/ 2$, by  \cite[Lemma 5]{Ng-Superlensing}, 
\begin{align}\label{thm1-p0-1}
\|(E_n, H_n) \|_{L^2(B_{R} \setminus B_{R_0 + 1/2})} &\le C_R \| (\nu \times E_n, \nu \times H_n) \|_{H^{-1/2}(\operatorname{div}_{\Gamma},\partial B_{R_0 + 1/2})} \notag \\[6pt]
&\le C_{R}\|(E_n, H_n ) \|_{H(\curl, B_{R_0 + 1/2} \setminus B_{R_0})} \mbox{ by the trace theory (see e.g. \cite{Girault})}. 
\end{align}
Combining  \eqref{thm1-contradiction}  and \eqref{thm1-p0-1} yields  
\begin{equation}\label{thm1-p0}
\|(E_{n}, H_{n}) \|_{L^2(B_R)} \le C_R \qquad \text{ for all } R > 0. 
\end{equation}

Fix $\psi \in C^1_{c}(\mR^3)$ (arbitrary) be such that $\psi = 0 $ in a neighborhood of $\Gamma$. We have, in $\mR^3$,  
\begin{equation*}
\nabla \times (\psi E_n) = \psi \nabla \times E_n + \nabla \psi \times E_n \quad \mbox{ and }  \quad \dive ( \eps_0 \psi E_n)  = \psi \dive (\eps_0 E_n) + \nabla \psi \cdot  \eps_0 E_n.  
\end{equation*}
Using \eqref{thm1-p0} and the system of equations of $(E_n, H_n)$ and applying e.g. \cite[Lemma 1]{Ng-Superlensing-Maxwell}, one may assume that 
$$
(\psi E_n) \mbox{ converges  in } L^2(\mR^3), 
$$ 
which yields,  since $\psi$ is arbitrary, 
\begin{equation}\label{thm1-p1}
(E_n) \mbox{ converges  in } L^2_{\loc}(\mR^3 \setminus \Gamma). 
\end{equation}
Similarly, one may also assume that 
\begin{equation}\label{thm1-p2}
(H_n) \mbox{ converges  in } L^2_{\loc}(\mR^3 \setminus \Gamma).  
\end{equation}
 
For $\tau > 0$ sufficiently small (the smallness depends only on $D$), define  ${\mathcal F}: D_{\tau} \to D_{-\tau}$ by 
\begin{equation*}
{\mathcal F}(x_\Gamma + t \nu(x_\Gamma)) = x_\Gamma - t \nu(x_\Gamma) \quad \forall \, x_\Gamma \in \Gamma, \, t \in (-\tau, 0) 
\end{equation*}
and set,  in  $D_{-\tau}$,  
\begin{equation*}
(\heps, \hmu) =  ({\mathcal F}_*\en, {\mathcal F}_*\mun) \quad \mbox{ and } \quad   (\hE_n, \hH_n)  = ({\mathcal F}*E_n, {\mathcal F}*H_n). 
\end{equation*}
Note that the pairs $( \eps, \heps )(x)$ and $( \mu, \hmu )(x)$ satisfy the complementing conditions with respect to $\nu(x)$ for 
all $x \in \Gamma$ if and only if the pairs $( \eps^{+}, -  \eps^{-} )(x)$ and $( \mu^{+}, - \mu^{-} )(x)$ do, by Proposition~\ref{pro-C}. By a change of variables for the Maxwell equations, see e.g.  \cite[Lemma 7]{Ng-Superlensing-Maxwell}, we have 
\begin{equation*}
\left\{ \begin{array}{cl}
\nabla \times E_n = i \omega \mu H_n  & \mbox{ in } D_{-\tau}, \\[6pt]
\nabla \times H_n = - i \omega \eps E_n + J_{n}  & \mbox{ in } D_{-\tau}, 
\end{array}\right. \quad \quad 
\left\{ \begin{array}{cl}
\nabla \times \hE_n = i \omega  \hmu \hH_n + \hJ_{e, n} & \mbox{ in } D_{-\tau}, \\[6pt]
\nabla \times \hH_n = - i \omega  \heps \hE_n + \hJ_{m, n} & \mbox{ in } D_{-\tau}, 
\end{array}\right. 
\end{equation*}
and 
\begin{equation*}
E_n \times \nu = \hE_n \times \nu  \quad \mbox{ and } \quad H_n \times \nu = \hH_n \times \nu \quad  \mbox{ on } \Gamma. 
\end{equation*}
Here 
$$
\hJ_{e, n} =  - \delta \omega   {\mathcal F} _*I  \, \hH_n \quad \mbox{ and } \quad \hJ_{m, n} =  \delta \omega {\mathcal F}_*I \, \hE_n + {\mathcal F}_*J_{n} \quad  \mbox{ in } D_{- \tau}.  
$$

By Lemma~\ref{lem-compact-1}, there exists a neighborhood $U$ of $\Gamma$ such that,  up to a subsequence, 
\begin{equation}\label{thm1-p2-*}
(E_n, H_n, \hE_n, \hH_n) \mbox{ converges in } [L^2(U \cap D_{-\tau/2})]^{12}.
\end{equation}
It follows from \eqref{thm1-p1}, \eqref{thm1-p2}, and \eqref{thm1-p2-*} that, up to a subsequence,  
\begin{equation}\label{strongconv}
(E_n, H_n) \mbox{ converges in } [L^2_{\loc} (\mR^3)]^{12}.
\end{equation}
Moreover, the limit $(E_0, H_0)$ is in $[H_{\loc}(\curl, \mR^3)]^2$ and is a radiating solution of the equations 
\begin{equation*}
\left\{\begin{array}{cl}
\nabla \times E_0 = i \omega \mu_0 H &  \mbox{ in } \mR^3, \\[6pt]
\nabla \times H_0 = - i \omega \eps_0 E&  \mbox{ in } \mR^3.  
\end{array} \right.
\end{equation*}
By Lemma~\ref{lem-unique-1}, we have
$$
E_0 = H_0 = 0 \mbox{ in } \mR^3. 
$$
This contradicts the fact that $\|(E_n, H_n)\|_{L^2(B_{r_0})}  = 1$ and, up to a subsequence,  $(E_n, H_n)$ converges to $(E_0, H_0)$ in $[L_{\loc}^2(\mR^3)]^6$. 

Hence \eqref{T1-1} holds for $R = R_0 + 1.$ This implies, exactly as in the proof of \eqref{thm1-p0},   
\begin{equation*}
 \|(E_{\delta}, H_{\delta}) \|_{L^2(B_{R} )} \le C_R \| J \|_{L^2(\mR^3)} \qquad \text{ for all } R > R_0 + 1,  
\end{equation*}
which is  \eqref{T1-1}. 

We next establish the remaining part of Theorem \ref{thm1}. By \eqref{T1-1}, for every sequence $(\delta_n) \to 0$, there exists a 
subsequence $(\delta_{n_k})$ such that $(E_{\delta_{n_k}}, H_{\delta_{n_k}})$ converges weakly in $[L^2_{\loc}(\mR^3)]^6$. Moreover, 
the limit $(E_0, H_0) \in [H_{\loc}(\curl, \mR^3)]^2$ is a radiating solution of \eqref{Main-eq-delta} with $\delta = 0$. By Lemma~\ref{lem-unique-1}, the limit is unique. 
Therefore, $(E_\delta, H_\delta)$ converges to $(E_0, H_0)$ weakly in $[L^2_{\loc}(\mR^3)]^6$ as $\delta \to 0$. It is then clear that  \eqref{T1-2} follows from \eqref{T1-1}. The strong 
convergence follows from \eqref{T1-1} by the same way 
\eqref{strongconv} followed from \eqref{thm1-p0} and once again using the uniqueness of the limit. The proof is complete. 
\proofend

\medskip 

A key part in the proof of Theorem~\ref{thm1} is the compactness result in Lemma~\ref{lem-compact-1} which is derived from Lemma~\ref{lem-compact-1-*} where the complementing conditions plays a crucial role. Assume that  $(\eps, \heps)$ and $(\mu, \hmu)$ both satisfy the complementing conditions and of class $C^2$. Then,  instead of \eqref{lem-compact-1-E-2} and \eqref{lem-compact-1-H-2},  one has (see \cite{ADNII})
 \begin{align}\label{uniformestimate}
  \big\|  ( E, \hE, H, \hH)  \big\|_{H^{1}(\Omega)} \leq C \left( 
   \big\| ( E, \hE, H, \hH) \big\|_{L^{2}(\Omega )}
   + \big\| ( J_{e}, \hJ_{e}, J_{m}, \hJ_{m}) \big\|_{H(\dive,\Omega)} \right) 
 \end{align}
for some positive constant $C$ independent of $(E, \hE, H, \hH)$, $(J_e, \hJ_e, J_m, \hJ_m)$.   The following proposition shows that the complementing conditions on each pair $(\eps, \heps)$ and $(\mu, \hmu)$ are necessary to have \eqref{uniformestimate}.

\begin{proposition}\label{pro-opt} Let $\Omega$ be a bounded, connected open subset of $\mR^3$ and of class $C^2$ and let $\eps, \; \heps, \; \mu, \; \hmu$  be real,  symmetric,   uniformly elliptic matrix-valued functions
defined  in $\Omega$ and of class $C^1$. 
Assume that  either $(\eps, \heps)$ or $(\mu,\hmu)$ does not satisfy the complementing condition at some point $x_{0} \in \partial\Omega$. Then there exist sequences  
 $(J_{e, n}), (J_{m, n}),  (\hJ_{e, n}), (\hJ_{m, n})  \subset  H(\operatorname*{div}, \Omega)$  and  $\big((E_n, H_n) \big) ,  \big( (\hE_n, \hH_n) \big) \subset [H^{1}(\Omega)]^6$ such that  
 \begin{equation*}
\left\{ \begin{array}{cl}
\nabla \times E_n = i \omega \mu H_n + J_{e, n} & \mbox{ in } \Omega, \\[6pt]
\nabla \times H_n = - i \omega \eps E_n + J_{m, n}  & \mbox{ in } \Omega, 
\end{array}\right. \quad \quad 
\left\{ \begin{array}{cl}
\nabla \times \hE_n = i \omega \hmu \hH_n + \hJ_{e, n} & \mbox{ in } \Omega, \\[6pt]
\nabla \times \hH_n = - i \omega \heps \hE_n + \hJ_{m, n} & \mbox{ in } \Omega, 
\end{array}\right. 
\end{equation*}
\begin{equation*}
E_n \times \nu = \hE_n \times \nu,  \quad H_n \times \nu = \hH_n \times \nu  \mbox{ on } \partial \Omega 
\end{equation*}
$$
\sup_{n} \|(J_{e, n}), (J_{m, n}),  (\hJ_{e, n}), (\hJ_{m, n})\|_{H(\dive(\Omega))} < + \infty,  \quad \sup\limits_{n} \lVert (E_n, \hE_n, H_{n}, \hH_{n}) \rVert_{L^{2}(\Omega)}< + \infty,
$$
and
$$
 \sup\limits_{n} 
 \lVert (E_n, \hE_n, H_{n}, \hH_{n}) \rVert_{H^{1}(\Omega)} = + \infty.  
 $$ 
\end{proposition}

The proof of Proposition~\ref{pro-opt} is given in the appendix.

\section{Variational approach  for the Cauchy problems}\label{sect-thm2}

This section containing three subsections  is devoted to Theorem~\ref{thm2}. In the first section, we present lemmas  used in the proof of Theorem~\ref{thm2}. The proof of Theorem~\ref{thm2} is given in the second subsection. In the last section, we give some applications of Theorem~\ref{thm2}. 

\subsection{Some useful lemmas} We begin with the following version of Poincar\'e's lemma which plays an important role in the proof of Theorem~\ref{thm2}.

\begin{lemma}\label{lem-extension-2} Let $0 \le \alpha < 2$ and $\Omega$ be a domain diffeomorphic to $B_1$. There exists a linear operator 
$$
T_1: \Big\{f \in H(\dive, \Omega); \,  \dive f = 0 \mbox{ in } \Omega \Big\} \to H(\curl, \Omega) 
$$
such that 
\begin{equation*}
\nabla \times T_1(f) = f \mbox{ in } \Omega
\end{equation*}
and 
\begin{equation} \label{lem-extension-2-2}
\int_{\Omega} |T_1 (f)|^2 \, dx  \le C \int_{\Omega}  d_{\partial \Omega}^\alpha |f|^2 \, dx, 
\end{equation}
for some positive constant $C = C(\alpha, \Omega)$ independent of $f$. 
\end{lemma}

\begin{proof} We first consider the case $\Omega = B_1$. Without loss of generality, one can assume that $1 < \alpha < 2$.  We first assume that $f$ is smooth (and $\dive f =0$).  Set 
\begin{equation}\label{def-F}
F (x) = - \int_0^1 t x \times f(t x) \,  dt \mbox{ in } B_1. 
\end{equation}
Using the fact that, for two vector fields $A$ and  $B$ defined in $B_1$,  
$$
\nabla \times (A \times B) = A \dive B - B \dive A  + (B \cdot \nabla ) A - (A \cdot \nabla ) B, 
$$
we have, in $B_1$, 
\begin{equation*}
\nabla \times F (x) =  2 \int_0^1 t f(tx) \, dt + \int_0^1 t (x \cdot \nabla ) f(tx) \, d t =  2 \int_0^1 t f(tx) \, dt + \int_0^1 t^2 \frac{d}{dt} [f(tx)] \, d t = f(x). 
\end{equation*}

We next  prove \eqref{lem-extension-2}.  We have 
\begin{align*}
\int_{B_1} |F|^2 \, dx & \mathop{\le}^{\eqref{def-F}} \int_{B_1} \left( \int_{0}^1 |tx| |f(tx)| \, dt \right)^{2} \, d x \\[6pt] &  \mathop{\le}^{\mbox{{\tiny H\"older}}} \int_{B_1} \left( \int_{0}^1  |tx|^2 |f(tx)|^2 (1 - |tx|)^{\alpha} \, dt \int_{0}^1 (1 - t |x|)^{-\alpha}  \,dt \right) \, d x. 
\end{align*}
In what follows in this proof, $C$ denotes a positive constant depending only on $\alpha$ and $\Omega$.  Since 
\begin{equation*}
\int_0^1 (1 - t |x|)^{-\alpha}  \,dt =  (\alpha -1) |x|^{-1} \Big[ (1 - |x|)^{-\alpha+1} - 1 \Big] \le C(1 - |x|)^{-\alpha+1}, 
\end{equation*}
it follows that 
\begin{align}\label{lem-extension-p1}
\int_{B_1} |F|^2 \, dx  
&\le C \int_{B_1}  \int_{0}^1 |tx|^2  |f(tx)|^2 (1 - |tx|)^{\alpha} \, dt  (1 - |x|)^{-\alpha + 1}   \, d x \\[6pt]
& \mathop{\le}^{y = t x} C \int_{0}^1  \int_{B_{t}} |y|^2 |f(y)|^2 (1 - |y|)^{\alpha}   (t - |y|)^{-\alpha + 1}  t^{\alpha-4} \, d y \, dt  \nonumber  \\[6pt]
& \mathop{\le}^{\mathrm{Fubini}} C \int_{B_{1}}  |y|^2 |f(y)|^2 (1 - |y|)^{\alpha}   \int_{y}^{1} (t - |y|)^{-\alpha + 1}  t^{\alpha - 4} \, d t \, dy  \nonumber . 
\end{align}
Since $\alpha < 2$, we have 
\begin{align}\label{lem-extension-p2}
\int_{y}^{1} (t - |y|)^{-\alpha + 1}  t^{\alpha - 4} \, d t = &  \frac{1}{2 - \alpha}(t - |y|)^{-\alpha + 2} t^{\alpha - 4}\Big|_y^1 +  \frac{1}{(2 - \alpha)(4 - \alpha)} \int_{y}^{1}(t - |y|)^{-\alpha + 2} t^{\alpha - 5} \, dt   \\[6pt]
\le & C + C \int_y^1 t^{-3} \le C y^{-2}.  \nonumber
\end{align}
Combining \eqref{lem-extension-p1} and \eqref{lem-extension-p2} yields 
\begin{equation}\label{lem-extension-2-2-1}
\int_{B_1} |F|^2 \, dx
\le C \int_{B_{1} }  |f(y)|^2 (1 - |y|)^{\alpha}  \, dy, 
\end{equation}
which is \eqref{lem-extension-2-2}.  The conclusion for a general $f$ now follows by a standard approximation process.

\medskip 
We now consider a general $\Omega$. 
Let ${\mathcal T}: \Omega \to B_1$ be a diffeomorphism.  Set, for $x' \in B_1$,  
$$
f'(x') = \frac{\nabla {\mathcal T}(x)}{J(x)} f(x) \mbox{ with } x' = {\mathcal T} (x) \mbox{ and } J(x) = \det {\mathcal T}(x)   
$$
and 
$$
F'(x') = - \int_0^1 t x' \times f' (t x') \,  dt. 
$$
Since  (see, e.g.,   \cite[Lemma 3.59]{monk-book-maxwell}), 
$$
\nabla' \cdot f'(x') = \frac{\nabla \cdot f(x)}{J(x)} = 0 \mbox{ in } B_1, 
$$
it follows from the case $\Omega = B_1$ that 
$$
\nabla' \times F' (x') = f'(x') \mbox{ in } B_1.  
$$
Set, for $x \in \Omega$,  
$$
T_1(f)(x) = \nabla {\mathcal T}^{T}(x) F'(x') \mbox{ with } x' = {\mathcal T} (x). 
$$
Then (see, e.g.,  \cite[Corollary 3.58]{monk-book-maxwell})
$$
\nabla \times T_1(f)(x) = f(x) \mbox{ for } x \in \Omega. 
$$
Assertion~\eqref{lem-extension-2-2} now follows from \eqref{lem-extension-2-2-1}. 
\end{proof}

\begin{remark} \rm The  definition of $F$ in \eqref{def-F} is an explicit formula for solving the Poincar\'{e} lemma in a star-shaped domain taken from a note of Dacorogna \cite{Dacorogna1}. 
\end{remark}

In what follows, we denote $\mR^3_+  = \{(x_1, x_2, x_3) \in \mR^3; x_3 > 0\}$, 
$\mR^3_0=   \{(x_1, x_2, x_3) \in \mR^3; x_3 = 0\}$,  $Q = (-1, 1)^3$, $Q_{+} = Q \cap \mR^3_+$, and  $Q_{0}= Q  \cap \mR^3_0$.  The proof of Lemma~\ref{lem-extension-2} can be used to obtain 

\begin{lemma}\label{lem-extension-3} Let $0 \le \alpha < 2$ and let $\Omega$ be an open bounded subset of $\mR^3$ and let ${\mathcal H} : \Omega \to Q_+$ be a diffeomorphism.  Denote $S = \partial Q_+  \setminus Q_0$ and let $U \subset \mR^3$ be a neighborhood of $\bar S$. 
There exist  an open set $V \subset \mR^3$ containing ${\mathcal H}^{-1}(\bar S)$ and a linear operator 
$$
T_1: \Big\{f \in H(\dive, \Omega); \,  \dive f = 0 \mbox{ in } \Omega \Big\} \to H(\curl, \Omega) 
$$
such that 
\begin{equation*}
\nabla \times T_1(f) = f \mbox{ in } \Omega, 
\end{equation*}
\begin{equation} \label{lem-extension-2-2}
\int_{\Omega} |T_1 (f)|^2 \, dx  \le C \int_{\Omega}  d_{\partial \Omega}^\alpha |f|^2 \, dx, 
\end{equation}
for some positive constant $C = C(\alpha, \Omega)$ independent of $f$,  and 
\begin{equation*}
T_1(f) = 0 \mbox{ in } V \cap \Omega \mbox{ if }  f= 0 \mbox{ in } {\mathcal H}^{-1} (U \cap Q_+). 
\end{equation*}
\end{lemma}

\begin{proof} Without loss of generality, one can assume that $1 < \alpha < 2$. Let  ${\mathcal T}_1: Q_+ \to B_1$ be a bi-Lipschitz homeomorphism  such that ${\mathcal T}_1 (S) = \Sigma$  and $\bar {\mathcal C} \subset {\mathcal T}_1( U \cap Q_+)$, where $\Sigma  :=  \partial B_1 \cap  \bar \mR^3_+$ and  ${\mathcal C} : =  \big\{r \sigma; r \in (0, 1) \mbox{ and } \sigma \in \Sigma \big\}$. Let $\Sigma_1 \subset \partial B_1$ be an open set of $ \partial B_1$ such that $\bar \Sigma \subset \Sigma_1$ and  $\bar {\mathcal C}_1 \subset {\mathcal T}_1( U \cap Q_+)$, where ${\mathcal C}_1 =  \big\{r \sigma; r \in (0, 1) \mbox{ and } \sigma \in \Sigma_1 \big\}$.

 Let ${\mathcal T}: \Omega \to B_1$ be defined by 
$$
{\mathcal T} = {\mathcal T}_1 \circ {\mathcal H}. 
$$
As in the proof  of Lemma~\ref{lem-extension-2}, set, for $x' \in B_1$,  
$$
f'(x') = \frac{\nabla {\mathcal T}(x)}{J(x)} f(x) \mbox{ with } x' = {\mathcal T} (x) \mbox{ and } J(x) = \det {\mathcal T}(x), 
$$
and 
\begin{equation}\label{def-F'}
F'(x') = - \int_0^1 t x' \times f' (t x') \,  dt. 
\end{equation}
Define, in $\Omega$, 
$$
T_1(f)(x) = \nabla {\mathcal T}^{T}(x) F'(x') \mbox{ with } x' = {\mathcal T} (x). 
$$
As in the proof of Lemma~\ref{lem-extension-2}, we have
\begin{equation*}
\nabla \times T_1(f) = f \mbox{ in } \Omega
\end{equation*}
and 
\begin{equation*}
\int_{\Omega} |T_1 (f)|^2 \, dx  \le C \int_{\Omega}  d_{\partial \Omega}^\alpha |f|^2 \, dx. 
\end{equation*}

Let $V$ be an open subset of $\mR^3$ such that ${\mathcal H}^{-1}(\bar S) \subset V$ and $V \cap \Omega = {\mathcal H}^{-1} ({\mathcal T}^{-1}({\mathcal C}_1))$. Such a $V$ exists by the choice of ${\mathcal C_1}$. 
Assume that $f = 0 $ in ${\mathcal H}^{-1} (U \cap Q_+)$. Then $f'(x') = 0$ for $x' \in {\mathcal C}_1$. It follows from \eqref{def-F'} that $F'(x') = 0$ for  $x' \in {\mathcal C}_1$.  This implies that $T_1(f) = 0$ in ${\mathcal H}^{-1} ({\mathcal T}^{-1}({\mathcal C}_1)) = V \cap \Omega$. The proof is complete. 
\end{proof}

In what follows, we use several times the following simple compactness result whose proof is omitted.  

\begin{lemma} \label{lem-Inter} Let $\Omega$ be a bounded open subset of $\mR^d$ of class $C^1$ and  let $\gamma_1 < \gamma_2$. Assume that $(u_n) \subset L^2_{\loc}(\Omega)$ is such that $u_n \to u$ in $L^2_{\loc}(\Omega)$ as $n \to + \infty$ and 
$$
\sup_{n} \| u_n\|_{L^2(d_{\partial \Omega}^{\gamma_1}, \Omega)} < + \infty. 
$$
Then $u_n \to u $ in $L^2(d_{\partial \Omega}^{\gamma_2}, \Omega)$ as $n \to + \infty$. 
\end{lemma}

Here and in what follows, for a non-negative,  measurable function $g$ defined in $\Omega$, one denotes, for $u \in L^2_{\loc}(\Omega)$, 
$$
\| u\|_{L^2(g, \Omega)} = \left( \int_{\Omega} g |u|^2 \right)^{1/2}. 
$$
For $(u_n) \subset L^2_{\loc}(\Omega)$ and $u \in L^2_{\loc}(\Omega)$, one says that $u_n \to u$ in $L^2(g, \Omega)$ as $n \to + \infty$ if $\lim_{n \to + \infty} \| u_n - u \|_{L^2(g, \Omega)} = 0$. 

Let $\Omega$ be a bounded open subset of $\mR^3$ with Lipschitz boundary. We denote 
\begin{equation*}
H_{0} (\curl, \Omega) = \Big\{u \in H(\curl, \Omega); u \times \nu = 0 \mbox{ on } \partial \Omega \Big\}.   
\end{equation*}
In the proof of Lemma~\ref{lem-M2}, we also use the following form of Poincar\'e's lemma:

\begin{lemma} \label{lem-curl-2}  Let  $\Omega$ be a  simply connected,  bounded, open subset of $\mR^3$ of class $C^2$. There exists a linear continuous transformation $T_0: H_{0} (\curl, \Omega) \to [H_0^1(\Omega)]^3$ such that 
\begin{equation*}
\nabla \times T_0(u)  = \nabla \times  u \mbox{ in } \Omega.
\end{equation*}
Consequently, for every $s > 0$, for every $0 \le \alpha < 2$, and for every $\beta \ge 0$,  there exists a constant $C_s = C(s, \alpha, \beta,  \Omega) > 0$ such that 
\begin{equation}\label{lem-curl-2-claim}
\| T_0(u) \|_{L^2(d_{\partial \Omega}^{-\alpha} \Omega)} \le s \|  u \|_{H(\curl, \Omega)} + C_s \| u \|_{L^2(d_{\partial \Omega}^\beta, \Omega)}. 
\end{equation}
\end{lemma}

\begin{proof} We only prove \eqref{lem-curl-2-claim} and we prove it by contradiction.  Assertion \eqref{lem-curl-2} is known, see e.g. \cite[Theorem 3.6]{Girault}, \cite[Theorem 8.16]{CsatoDacKneuss}. 
Assume that \eqref{lem-curl-2-claim} does not hold. There exists a sequence
of $(u_n) \subset H_0(\curl,  \Omega)$ such that 
\begin{equation}\label{lem-curl-2-e1}
1= \| T_0(u_n) \|_{L^2(d_{\partial \Omega}^{-\alpha}, \Omega)} \ge s \|  u_n \|_{H(\curl, \Omega)}^2 + n \| u_n \|_{L^2(d_{\partial \Omega}^\beta, \Omega)}^2. 
\end{equation}
Since $T_0$ is continuous, it follows that  the sequence 
$ \big( \| T_0(u_n) \|_{H^1(\Omega)} \big)$ is bounded. Without loss of generality, one may assume that $T_0(u_n) \to v$ in $[L^2(\Omega)]^3$ and weakly in $[H^1(\Omega)]^3$. By Hardy's inequality
\begin{equation*}
\int_{\Omega} d_{\partial \Omega}^2 |v |^{-2} \le C_{\Omega} \int_{\Omega}  |\nabla v |^2, 
\end{equation*}
we derive that
\begin{equation*}
\int_{\Omega} d_{\partial \Omega}^{-2} |T_0(u_n) |^2 \le C. 
\end{equation*}
Since $0 \le \alpha < 2$, by Lemma~\ref{lem-Inter}, without loss of generality, one may assume that $T_0(u_n) \to v$ in $[L^2(d_{\partial \Omega}^{-\alpha}, \Omega)]^3$.

From \eqref{lem-curl-2-e1}, one may assume as well that $u_n \rightharpoonup 0$ weakly in $H_0(\curl, \Omega)$. Since $T_0$ is linear and continuous, this in turn implies that $v = T_0 ( \lim_{n \to + \infty}  u_n)  = T_0 (0) = 0 $. This contradicts  the fact $\|v \|_{L^2(d_{\partial \Omega}^{-\alpha}, \Omega)} = \lim_{n \to + \infty}\| T_0(u_n) \|_{L^2(d_{\partial \Omega}^{-\alpha}, \Omega)} = 1$. Therefore, \eqref{lem-curl-2-claim} holds.  \end{proof}

We now  present the key ingredient of the proof of Theorem~\ref{thm2}, which is a variant of Lemma~\ref{lem-compact-1}. 

\begin{lemma}\label{lem-M2} Let  $\tau > 0$, $0 < \delta < 1$,  $s>0$, and let $\Omega$ be a connected component of $D_{-\tau}$. Let $\eps, \; \heps, \; \mu, \; \hmu$  be real,  symmetric,   uniformly elliptic matrix-valued functions
defined   in $\Omega$. 
Let   $J_{e}, J_{m},  \hJ_{e}, \hJ_{m}  \in [L^2(\Omega)]^3$  and  $(E, H), (\hE, \hH) \in [H(\curl, \Omega)]^2$ be such that 
 \begin{equation}\label{lem-M2-sys}
\left\{ \begin{array}{cl}
\nabla \times E = i \omega    \mu H + J_{e} & \mbox{ in } \Omega, \\[6pt]
\nabla \times H = - i \omega  \eps E + J_{m}  & \mbox{ in } \Omega, 
\end{array}\right. \quad \quad 
\left\{ \begin{array}{cl}
\nabla \times \hE = i \omega \hmu  \hH + \hJ_{e} & \mbox{ in } \Omega, \\[6pt]
\nabla \times \hH = - i \omega   \heps  \hE + \hJ_{m} & \mbox{ in } \Omega, 
\end{array}\right. 
\end{equation} 
and 
\begin{equation}\label{lem-M2-bdry}
E \times \nu = \hE \times \nu  \quad \mbox{ and } \quad H \times \nu = \hH \times \nu  \mbox{ on } \partial \Omega \cap \Gamma. 
\end{equation}
We have 
\begin{enumerate}
\item[1)] if, for some $c>0$ and $0 \le \alpha_1 < 2$, 
\begin{equation*}
\heps - \eps \ge c d_\Gamma^{\alpha_1} I   \mbox { in } \Omega  \quad \mbox{ or } \quad  \eps - \heps \ge c d_\Gamma^{\alpha_1} I    \mbox { in } \Omega
\end{equation*}
then, for every $s>0$ and for every $0 \le \beta  < 2$, there exists a constant $C_{s} > 0$, depending only on $s$, $\beta$, $c$,  $\alpha_1$,  $\Omega$, and  the ellipticity of $\eps, \,  \heps, \,  \mu, \,  \hmu$  such that 
\begin{multline}\label{lem-M2-assertion1}
\int_{\Omega} | \langle (\eps - \heps) E, E \rangle| + |E - \hE|^2 
 \le  s  \int_{\Omega} |\mu H - \hmu \hH|^2 + C_s  \int_{\Omega} |J_m - \hJ_m| | \hE|\\[6pt]
+  C_{s}  \int_{\Omega} d_{\Gamma}^{\beta} |(E, \hE, H, \hH)|^2  + |(J_{e}, \hJ_e, J_{m}, \hJ_m)|^2; 
\end{multline}
\item[2)] if, for some $c>0$ and $0 \le \alpha_2 < 2$, 
\begin{equation*}
\hmu - \mu \ge c d_\Gamma^{\alpha_1} I  \mbox { in } \Omega \quad \mbox{ or } \quad  \mu - \hmu  \ge c d_\Gamma^{\alpha_1} I  \mbox { in } \Omega
\end{equation*}
then, for every $s>0$ and for every $0 \le \beta  < 2$, there exists a constant $C_{s}> 0$, depending only on $s$, $\beta$,  $c$,  $\alpha_2$, $\Omega$, and  the ellipticity of $\eps, \, \heps, \, \mu, \, \hmu$ such that 
\begin{multline}\label{lem-M2-assertion2}
\int_{\Omega} | \langle (\mu - \hmu) H, H \rangle| + |H - \hH|^2
 \le s \int_{\Omega} |\eps E - \heps \hE|^2  + C_s  \int_{\Omega} |J_e - \hJ_e| | \hH|\\[6pt]
 + C_{s} \int_{\Omega} d_{\Gamma}^{\beta} |(E, \hE, H, \hH)|^2  + |(J_{e}, \hJ_e, J_{m}, \hJ_m)|^2.
\end{multline}
\end{enumerate}
\end{lemma}

\begin{proof}  We only establish assertion $1)$ and assume that $ \eps - \heps \ge c d_\Gamma^{\alpha_1} I$. The proof of assertion $1)$ in the case $ \heps -  \eps \ge c d_\Gamma^{\alpha_1} I $ and the proof of assertion $2)$ can be derived similarly.  Without loss of generality, one may assume that $s$ is small. 

We use local charts for $\partial \Omega \cap \Gamma$.  Let $\ell \ge 1$ and let   $\varphi_k \in  C^1_{c}(\mR^3)$, $U_k \subset \mR^3$  open ball, and ${\mathcal H}_k : U_k \to Q$ for $1 \le k \le \ell $ be such that ${\mathcal H}_k$ is a  diffeomorphism, ${\mathcal H}_k(U_k \cap \Omega) = Q_+$, and ${\mathcal H}_k(U_k \cap \Gamma) = Q_0$, $\supp \varphi_k \Subset U_k$, and $\Phi  = 1$ in a neighborhood of $\partial \Omega \cap \Gamma$, where 
$$
\Phi : = \sum_{k=1}^\ell \varphi_k \mbox{ in } \mR^3.
$$

 By Lemma~\ref{lem-extension-3}, for $1 \le k \le \ell$,  there exists $\hbE_k \in H(\curl, U_k \cap \Omega)$ and an open $ V_k \subset \mR^3$ containing $\overline{\partial U_k \cap  \Omega}$ 
 such that
\begin{equation}\label{lem-M2-1-1}
\nabla \times \hbE_k = \nabla \times (\varphi_k \hE)   \mbox{ in } U_k \cap  \Omega, 
\end{equation}
\begin{equation}\label{lem-M2-1-2}
\| \hbE_k \|_{L^2(U_k \cap \Omega)}  \le C \| \hE \|_{L^2(d_{\Gamma}^{\beta}, U_k \cap \Omega)}, 
\end{equation}
\begin{equation}\label{lem-M2-1-3}
\hbE_k = 0 \mbox{ in }  U_k \cap  \Omega \cap V_k.  
\end{equation}
Here and in what follows, $C$ denotes a positive constant depending only on $c$, $\alpha_1$, $\beta$, $\Omega$, $\eps, \, \heps, \, \mu$, and  $\hmu$; thus $C$ is independent of $s$.

Set, for $1 \le k \le \ell$,  
\begin{equation*}
\dpkE = \varphi_k( E - \hE) \mbox{ in } U_k \cap  \Omega, 
\end{equation*}
Let $s_1$ be a small positive number  defined later. 
Since $(E - \hE) \times \nu = 0$ on $\partial \Omega \cap \Gamma$, by Lemma~\ref{lem-curl-2}, for $1 \le k \le \ell$,  there exists $\bdkE \in H_0^1(U_k \cap \Omega)$ such that
\begin{equation}\label{lem-M2-2-1}
\nabla \times  \bdkE = \nabla \times \dpkE\mbox{ in } U_k \cap \Omega, 
\end{equation}
\begin{equation}\label{lem-M2-2-2}
\| \bdkE \|_{L^2(d_{\partial \Omega}^{-\beta}, U_k \cap \Omega)}  \le C_{s_1} \|\dpkE \|_{L^2(d_{\Gamma}^{\beta}, U_k \cap \Omega)} + s_1 \| \dpkE \|_{H(\curl, U_k \cap \Omega)}. 
\end{equation}
It is clear then 
\begin{equation}\label{lem-M2-2-3}
\bdkE = 0 \mbox{ in } \partial (U_k \cap  \Omega).   
\end{equation}

We have, in $\Omega$,  
\begin{align*}
\nabla \times (\varphi_k E - \varphi_k \hE)  &=   \nabla \varphi_k \times (E - \hE) + \varphi_k \nabla (E - \hE) \\[6pt]
& \mathop{=}^{\eqref{lem-M2-sys}}   \nabla \varphi_k \times (E - \hE) + i \omega \varphi_k (\mu H - \hmu \hH) + \varphi_k (J_e - \hJ_e).  
\end{align*}
We derive from \eqref{lem-M2-2-2} that 
\begin{equation}\label{lem-M2-2-4}
\| \bdkE \|_{L^2(d_{\partial \Omega}^{-\beta} U_k \cap \Omega)}  \le  C_{s_1}  \|E - \hE\|_{L^2(d_{\Gamma}^{\beta}, U_k \cap \Omega)}  + C s_1 \| \big(E - \hE, \mu H - \hmu \hH,  J_e - \hJ_e\big)\|_{L^2(U_k \cap \Omega)}. 
\end{equation}

For $1 \le k \le \ell$, let $\xi_k, \eta_k \in H^1(U_k \cap \Omega)$ be such that 
\begin{equation}\label{lem-M2-xi-eta}
\nabla \xi_k = \dpkE -  \bdkE \quad \mbox{ and } \quad \nabla \eta_k = \varphi_k \hE - \hbE_k \mbox{ in } U_k \cap \Omega. 
\end{equation}
Such $\xi_k$ and $\eta_k$ exist by Poincare's lemma, \eqref{lem-M2-1-1}, and \eqref{lem-M2-2-1}; moreover, one can assume that 
\begin{equation}\label{lem-M2-supp-xi-eta}
\xi_k  = 0 \mbox{ on }   \partial (U_k \cap  \Omega) \quad \mbox{ and } \quad  \eta_k =   0  \quad \mbox{ on }  \partial U_k \cap  \Omega
\end{equation}
since $\nabla \xi_k \times \nu = 0$ on  $\partial (U_k \cap  \Omega)$ by  \eqref{lem-M2-2-3} and  $ \nabla \eta_k \times \nu  = 0$ on $\partial U_k \cap  \Omega$ by \eqref{lem-M2-1-3}. 

Extend $\xi_k$ and $\eta_k$ by 0 in $\Omega \setminus U_k$ and still denote these extensions by $\xi_k$ and $\eta_k$. 
Set, in $\Omega$, 
\begin{equation}\label{lem-M2-def-FE}
F_E =  \Phi \left(\eps E - \heps \hE + \frac{i}{\omega} J_m - \frac{i}{\omega} \hJ_m \right) + \frac{i}{\omega} \nabla \Phi \times (H - \hH),
\end{equation}
\begin{equation}\label{lem-M2-delta-f}
\delta_f = \sum_{k = 1}^{\ell} \bdkE,  \quad f = \sum_{k=1}^{\ell} \hbE_k, \quad \mbox{ and }
 \quad 
\delta_{\Phi E} = \Phi ( E - \hE ). 
\end{equation}
It follows from \eqref{lem-M2-xi-eta} and \eqref{lem-M2-delta-f} that, in $\Omega$,  
\begin{equation}\label{lem-M2-xi-eta-sum}
\sum_{k=1}^\ell \nabla \xi_k = \delta_{\Phi E} - \delta_f \quad \mbox{ and } \quad \sum_{k=1}^\ell \nabla \eta_k = \Phi \hE - f, 
\end{equation}
and from \eqref{lem-M2-1-2} and \eqref{lem-M2-2-4} that 
\begin{multline}\label{lem-M2-est-delta-f}
\| \delta_f\|_{L^2(d_{\partial \Omega}^{-\beta}, \Omega)} + 
\| f \|_{L^2(\Omega)} \le C_{s_1} \| (J_e, \hJ_e)\|_{L^2(\Omega)} \\[6pt]+ C_{s_1}  \| (E, \hE, H, \hH) \|_{L^2(d_{\Gamma}^{\beta}, \Omega)} 
+ C s_1 \| \big(E - \hE, \mu H - \hmu \hH \big)\|_{L^2(U_k \cap \Omega)}.
\end{multline}

Since, by \eqref{lem-M2-sys} and \eqref{lem-M2-def-FE}, 
$$
F_E =\frac{i}{\omega}  \nabla \times (\Phi H ) - \frac{i}{\omega} \nabla \times (\Phi \hH) \mbox{ in } \Omega
$$
and, by \eqref{lem-M2-bdry}, 
$$
\big(\nabla \times (\Phi H  - \Phi \hH) \big)  \cdot  \nu = \dive_{\Gamma} \big( (\Phi H  - \Phi \hH)  \times \nu \big) = 0 \mbox{ on } \partial \Omega \cap \Gamma, 
$$ 
we obtain 
\begin{equation}\label{lem-M2-s1-1}
\dive F_{E} = 0 \mbox{ in } \Omega \quad \mbox{ and } \quad F_E \cdot \nu = 0 \mbox{ on } \partial \Omega \cap \Gamma. 
\end{equation}

From \eqref{lem-M2-supp-xi-eta} and \eqref{lem-M2-s1-1}, we  have, for $1 \le k \le \ell$,  
\begin{equation*}
\int_{\Omega} \langle   F_E,  \nabla \eta_k   \rangle = 0.   
\end{equation*}
Summing this identity with respect to $k$ and 
using  \eqref{lem-M2-xi-eta-sum},  we obtain 
\begin{equation}\label{lem-M2-2}
\int_{\Omega} \langle   F_E,  \Phi \hE  \rangle = \int_{\Omega} \langle   F_E ,   f  \rangle. 
\end{equation}

We have, for $1 \le k \le \ell$,  
$$
\frac{i}{\omega} \big( \nabla \times (\varphi_k \hH) - \nabla \varphi_k \times \hH\big)  = \frac{i}{\omega} \varphi_k \nabla \times \hH \mathop{=}^{\eqref{lem-M2-sys}} \heps  (\varphi_k \hE) + \frac{i}{\omega} \varphi_k \hJ_m \mbox{ in } \Omega. 
$$
It follows that 
\begin{equation*}
\dive [\heps  (\varphi_k \hE) ]  =   - \frac{i}{\omega} \dive \big( \varphi_k \hJ_m + \nabla \varphi_k \times \hH \big) \mbox{ in } \Omega.  
\end{equation*}
We derive that, for $1 \le k, l \le \ell$,   
\begin{equation}\label{lem-M2-integration}
\int_{\Omega} \langle    \heps \nabla \xi_l, \varphi_k \hE   \rangle =   \frac{i}{\omega}  \int_{\Omega}  \langle   \nabla \xi_l,  \varphi_k \hJ_m + \nabla \varphi_k \times \hH \rangle, 
\end{equation}
since $\xi_l = 0$ on $\partial \Omega$.  Summing with respect to $k$ and $l$ and using \eqref{lem-M2-xi-eta-sum}, we get
\begin{equation*}
\int_{\Omega} \langle    \heps (\delta_{\Phi E} - \delta_f), \Phi \hE  \rangle =   \frac{i}{\omega}  \int_{\Omega}  \langle  \delta_{\Phi E} - \delta_f,  \Phi  \hJ_m + \nabla  \Phi \times \hH \rangle. 
\end{equation*}
This yields 
\begin{equation}\label{lem-M2-3}
\int_{\Omega} \langle    \heps \delta_{\Phi E} , \Phi \hE   \rangle =  \int_{\Omega} \langle    \heps \delta_f  , \Phi \hE \rangle +  \frac{i}{\omega}  \int_{\Omega}  \langle  \delta_{\Phi E} - \delta_f,  \Phi  \hJ_m +  \nabla  \Phi \times H \rangle. 
\end{equation}

Noting  that $\eps E - \heps \hE = \eps (E - \hE ) + (\eps - \heps) \hE$ in $\Omega$,    we obtain from \eqref{lem-M2-def-FE} that 
\begin{equation}\label{lem-M2-FE}
F_E = \eps \delta_{\Phi E} + (\eps - \heps) \Phi  \hE +  \frac{i}{\omega}  \Phi \left(J_m - \hJ_m \right) + \frac{i}{\omega} \nabla \Phi \times (H - \hH) \mbox{ in } \Omega. 
\end{equation}
Subtracting \eqref{lem-M2-3} from \eqref{lem-M2-2} and using \eqref{lem-M2-FE}, we have
\begin{multline}\label{lem-M2-4}
\int_{\Omega} \langle (\eps - \heps) \delta_{\Phi E}, \Phi \hE \rangle + \int_{\Omega} \langle (\eps - \heps)  \Phi \hE, \Phi \hE \rangle \\[6pt] 
= -  \frac{i}{\omega} \int_{\Omega} \langle  \Phi \left(J_m - \hJ_m \right) + \nabla \Phi \times (H - \hH), \Phi \hE \rangle \\[6pt]
 + \int_{\Omega} \langle   F_E,   f  \rangle - \int_{\Omega} \langle    \heps \delta_f  , \Phi \hE    \rangle
 - \frac{i}{\omega}  \int_{\Omega}  \langle  \delta_{\Phi E} - \delta_f,  \Phi  \hJ_m +\nabla  \Phi \times H \rangle. 
\end{multline}

From \eqref{lem-M2-supp-xi-eta} and  \eqref{lem-M2-s1-1}, we derive  that, for $1 \le k \le \ell$,
\begin{equation}\label{lem-M2-s1}
\int_{\Omega} \langle F_E,  \nabla  \xi_k \rangle  = 0. 
\end{equation}
Summing \eqref{lem-M2-s1} with respect to $k$ and using \eqref{lem-M2-xi-eta-sum}, we get 
\begin{equation*}
\int_{\Omega} \langle F_E,  \delta_{\Phi E} - \delta_f \rangle = 0. 
\end{equation*}
It follows from \eqref{lem-M2-FE} that 
\begin{multline}\label{lem-M2-1}
 \int_{\Omega} \langle \eps \delta_{\Phi E}, \delta_{\Phi E}  \rangle + \int_{\Omega} \langle [\eps - \heps] \Phi \hE, \delta_{\Phi E} \rangle \\[6pt] = -  \frac{i}{\omega} \int_{\Omega}  \langle \Phi (J_m - \hJ_m)  + \nabla \Phi \times (H - \hH), \delta_{\Phi E} \rangle +  \int_{\Omega} \langle F_E , \delta_{f} \rangle.  
\end{multline}

We have, by \eqref{lem-M2-FE}, 
\begin{multline}\label{lem-M2-5-0}
\left| \int_{\Omega} \langle   F_E,   f  \rangle  \right| \le C \left( \int_{\Omega} |f|^2 \right)^{1/2}  \left( \int_{\Omega} | \langle (\eps - \heps) \Phi \hE , \Phi \hE  \rangle| + |\delta_{\Phi E}|^2\right)^{1/2}  \\[6pt]
+ C \int_{\Omega} |(f, J_m, \hJ_m,  \nabla \Phi \times H, \nabla \Phi \times \hH )|^2.  
\end{multline}

Since $\eps - \heps \ge 0$ and $\eps$ is uniformly elliptic, we deduce that 
$$
\langle (\eps - \heps) x, x \rangle + \langle \eps y, y \rangle + \langle  (\eps - \heps) x, y \rangle + \langle  (\eps - \heps) y, x \rangle \ge C \big( \langle (\eps - \heps) x, x \rangle + \langle \eps y, y \rangle \big), 
$$
for any $x, y \in \mR^3$.  It follows from  \eqref{lem-M2-4}, \eqref{lem-M2-1}, and \eqref{lem-M2-5-0} that 
\begin{multline}\label{lem-M2-5}
\int_{\Omega} | \langle (\eps - \heps) \Phi \hE , \Phi \hE  \rangle| + |\dpE|^2 \\[6pt]
 \le C   \int_{\Omega}   |(J_m, \hJ_m, \nabla \Phi \times H, \nabla \Phi \times \hH, f, \delta_f)|^2 + C \Big|\int_{\Omega}\langle \heps \delta_f, \Phi E \rangle  \Big| \\[6pt]
 + C  \int_{\Omega} |J_m - \hJ_m| | \hE| +  C  \int_{\Omega} | \nabla \Phi \times (H - \hH)| |\Phi \hE|. 
\end{multline}

We next estimate $\Big| \int_{\Omega}\langle \heps \delta_f, \Phi E \rangle  \Big|$.   Since
\begin{multline*}
\left| \int_{\Omega}\langle \heps \delta_f, \Phi E \rangle  \right| =  \left| \int_{\Omega}\langle (\eps- \heps)^{-1/2} \heps  \delta_f,  (\eps- \heps)^{1/2} \Phi E \rangle  \right| \\[6pt]
 \le \frac{1}{4s}
  \int_{\Omega} \langle (\eps - \heps)^{-1/2}  \heps \delta_f, (\eps - \heps)^{-1/2}  \heps  \delta_f \rangle  + s \int_{\Omega} \langle  (\eps - \heps)^{1/2} \Phi E, (\eps - \heps)^{1/2} \Phi E \rangle
\end{multline*}
and $\eps - \heps \ge c d_{\Gamma}^{\alpha_1} I$ in $\Omega$, we obtain
\begin{equation}\label{lem-M2-interpolation}
\left| \int_{\Omega}\langle \heps \delta_f, \Phi E \rangle  \right|  \le  \frac{C}{s}
  \int_{\Omega} d_{\Gamma}^{-2 \alpha_1} |\delta_f|^2  + s \int_{\Omega} \langle  (\eps - \heps) \Phi E, \Phi E \rangle. 
\end{equation}
Combining  \eqref{lem-M2-est-delta-f} and \eqref{lem-M2-interpolation} yields
\begin{multline}\label{lem-M2-est-I}
\left| \int_{\Omega}\langle \heps \delta_f, \Phi E \rangle  \right| \le C_{s, s_1} \int_{\Omega} |(J_e, \hJ_e)|^2  + C_{s, s_1} \int_{\Omega} d_{\Gamma}^\beta |(E, \hE, H, \hH)|^2   \\[6pt]+ \frac{C s_1 }{s} \int_{\Omega} |(\mu H -\hmu \hH, E - \hE)|^2  
+  s \int_{\Omega} \langle  (\eps - \heps) \Phi E, \Phi E \rangle. 
\end{multline}
Using \eqref{lem-M2-est-delta-f} and the fact $\Phi = 1$ in a neighborhood of $\Gamma \cap \partial \Omega$, we derive  from \eqref{lem-M2-5} and \eqref{lem-M2-est-I} that 
\begin{multline}\label{lem-M2-est-II}
\int_{\Omega} | \langle (\eps - \heps) \Phi \hE , \Phi \hE  \rangle| + |\dpE|^2  \\[6pt]
 \le C_{s, s_1}   \int_{\Omega}   |(J_e, \hJ_e, J_m, \hJ_m)|^2
 + C_{s, s_1} \int_{\Omega} d_{\Gamma}^\beta |(E, \hE, H, \hH)|^2  \\[6pt]
+  \frac{C s_1}{s}  \int_{\Omega} |(\mu H -\hmu \hH, E - \hE)|^2 \\[6pt]
+ C s \int_{\Omega} \langle  (\eps - \heps) \Phi E, \Phi E \rangle + C  \int_{\Omega} |J_m - \hJ_m| | \hE|. 
\end{multline}
Take $s_1 = C s^2$.   One derives from \eqref{lem-M2-est-II} that, for $s$ small,  
\begin{multline*}
\int_{\Omega} | \langle (\eps - \heps) \Phi \hE , \Phi \hE  \rangle| + |\dpE|^2 
 \le C_{s}   \int_{\Omega}   |(J_e, \hJ_e, J_m, \hJ_m)|^2 + 
 C_{s} \int_{\Omega} d_{\Gamma}^\beta |(E, \hE, H, \hH)|^2  \\[6pt]
+   s  \int_{\Omega} |(\mu H -\hmu \hH)|^2 + C  \int_{\Omega} |J_e - \hJ_e| | \hH|, 
\end{multline*}
which implies \eqref{lem-M2-assertion1}. 
\end{proof}

\begin{remark}  \label{rem-proof-thm2} \rm The proof of Lemma~\ref{lem-M2} involves local charts. The involvement is  of a global character in the sense that one has to combine local charts before deriving desired estimates in some parts of the proof, see \eqref{lem-M2-5} and \eqref{lem-M2-est-II}. In fact,  one cannot derive variants of 
\eqref{lem-M2-assertion1} and \eqref{lem-M2-assertion2} for $(\varphi_k E, \varphi_k \hE, \varphi_k H, \varphi_k \hH)$. To this end, note that,  in $U_k \cap \Omega$, 
 \begin{equation*}
\left\{ \begin{array}{cl}
\nabla \times (\varphi_k E) = i \omega    \mu (\varphi_k H) + J_{e, k} , \\[6pt]
\nabla \times (\varphi_k H) = - i \omega  \eps (\varphi_k E) + J_{m, k}  , 
\end{array}\right. \quad \quad 
\left\{ \begin{array}{cl}
\nabla \times (\varphi_k \hE) = i \omega \hmu  (\varphi_k \hH) + \hJ_{e, k}  \\[6pt]
\nabla \times (\varphi_k \hH) = - i \omega   \heps  (\varphi_k \hE) + \hJ_{m, k} , 
\end{array}\right. 
\end{equation*} 
where
$$
J_{e, k} = \varphi_k J_e + \nabla \varphi_k \times E, \quad  \hJ_{e, k} = \varphi_k \hJ_e + \nabla \varphi_k \times \hE, 
$$
$$
J_{m, k} = \varphi_k \hJ_m + \nabla \varphi_k \times H, \quad  \hJ_{m, k} = \varphi_k \hJ_m + \nabla \varphi_k \times \hH. 
$$
Due to the terms $\nabla \varphi_k \times E$, $\nabla \varphi_k \times \hE$, $\nabla \varphi_k \times H$, and $\nabla \varphi_k \times \hH$ in $J_{e, k}, \, \hJ_{e, k}, \, J_{m, k}$,  and $\hJ_{m, k}$, respectively, we are not able to derive the variants of \eqref{lem-M2-assertion1} and \eqref{lem-M2-assertion2} for $(\varphi_k E, \varphi_k \hE, \varphi_k H, \varphi_k \hH)$.  This combination is the key difference between the proof strategies of Lemmas~\ref{lem-M2} and \ref{lem-compact-1} and makes the proof of Lemma~\ref{lem-M2} more involved. Another difference between the proofs is that one considers the extensions of $\varphi E_k$ and $\varphi_k (E - \hE)$ in the proof Lemma~\ref{lem-M2} instead of the extensions of $\varphi E_k$ and $\varphi_k \hE$ as in the proof Lemma~\ref{lem-compact-1} to ensure 
integration by parts arguments, see \eqref{lem-M2-integration} where $\xi_k = 0$ on $\partial (\Omega \cap U_k)$ is required. 
\end{remark}

\subsection{Proof of Theorem~\ref{thm2}}  For a simpler presentation, we will assume that $D_{-\tau}$ is connected.  We first consider the case where $\alpha_1 + \alpha_2 > 0$. Set 
\begin{equation*}
\beta = (  \max\{ \alpha_1, \alpha_2 \} + 2)/ 2
\end{equation*}
and 
$$
\dEd = E_\delta - \hE_\delta  \mbox{ in } D_{-\tau}  \quad \mbox{ and } \quad \dHd = H_\delta - \hH_\delta \mbox{ in } D_{-\tau}. 
$$
Then 
\begin{equation}\label{thm2-def-beta}
\max\{ \alpha_1,  \alpha_2 \}  < \beta  < 2. 
\end{equation}

We first prove by contradiction that 
\begin{equation}\label{thm2-p1}
\| (E_\delta, H_\delta) \|_{L^2(d^{\beta}_{\Gamma},  B_{R_0})}  \le C \| J \|_{L^2(\mR^3)}.   
\end{equation}
Assume, for some $\delta_n \to 0$ and $J_n \in [L^2(\mR^3)]^3$ with $\supp J_n \subset B_{R_0} \setminus 
\big(\big(D_{-\tau} \cup {\mathcal F}^{-1}(D_{-\tau}) \big) = \emptyset$ that 
\begin{equation}\label{thm2-contradiction-1}
\lim_{n \to + \infty} \| J_n\|_{L^2(\mR^3)} = 0 
\quad \mbox{ and } \quad   \| (E_n, H_n) \|_{L^2(d^{\beta}_\Gamma,  B_{R_0})}
 = 1, 
\end{equation}
where $(E_n, H_n)$ is the solution corresponding to $\delta_n$ and $J_n$. Integrating by parts and 
using the radiating condition, as usual, we obtain 
\begin{equation*}
\left| \Im \int_{B_{R_0}} \langle \mu_{\delta_n}^{-1} \nabla \times E_n, \nabla \times E_n \rangle - \omega^2 \langle \eps_{\delta_n} E_n, E_n \rangle \right|  \le \left| \int_{B_{R_0}} \langle  i \omega J_n, E_n \rangle \right| \to 0 \mbox{ as } n \to + \infty. 
\end{equation*}
This implies 
\begin{equation}\label{thm2-p0-111}
 \int_{D} \delta_n ( |E_n|^2 + |H_n|^2 ) \to 0 \mbox{ as } n \to + \infty. 
\end{equation}
By a change of variables for the Maxwell equations, see e.g.  \cite[Lemma 7]{Ng-Superlensing-Maxwell}, we have 
\begin{equation*}
\left\{ \begin{array}{cl}
\nabla \times E_n = i \omega \mu H_n  & \mbox{ in } D_{-\tau}, \\[6pt]
\nabla \times H_n = - i \omega \eps E_n & \mbox{ in } D_{-\tau}, 
\end{array}\right. \quad \quad 
\left\{ \begin{array}{cl}
\nabla \times \hE_n = i \omega  \hmu \hH_n + \hJ_{e, n} & \mbox{ in } D_{-\tau}, \\[6pt]
\nabla \times \hH_n = - i \omega  \heps \hE_n + \hJ_{m, n} & \mbox{ in } D_{-\tau}, 
\end{array}\right. 
\end{equation*}
and 
\begin{equation*}
E_n \times \nu = \hE_n \times \nu  \quad \mbox{ and } \quad H_n \times \nu = \hH_n \times \nu \quad  \mbox{ on } \Gamma. 
\end{equation*}
Here 
$$
\hJ_{e, n} =  - \delta \omega   {\mathcal F} _*I  \, \hH_n \quad \mbox{ and } \quad \hJ_{m, n} =  \delta \omega {\mathcal F}_*I \, \hE_n \quad  \mbox{ in } D_{- \tau}.  
$$
Note that, in $D_{-\tau}$,  
$$
\eps E - \heps \hE = (\eps  - \heps) E  + \heps (E - \hE)
$$
and 
$$
\mu H - \hmu \hH =  (\mu  - \hmu) H  + \hmu (H - \hH). 
$$
Applying \eqref{lem-M2-assertion1} and \eqref{lem-M2-assertion2} for a sufficiently small $s$ and using \eqref{thm2-p0-111}, we  have
\begin{equation}\label{thm2-p2-1}
\int_{D_{-\tau}} |\langle (\eps - \heps) E_n, E_n  \rangle| + \int_{D_{-\tau}} |\langle (\mu - \hmu) H_n, H_n  \rangle + \int_{D_{-\tau}} |(\dEn, \dHn)|^2  \le C. 
\end{equation}
This implies 
\begin{equation}\label{thm2-p2}
\| E_n \|_{L^2(d_\Gamma^{\alpha_1},  B_{R_0})} + \| H_n \|_{L^2(d_\Gamma^{\alpha_2},  B_{R_0})}  \le C. 
\end{equation}
Fix $\psi \in C^1_{c}(\mR^3)$ (arbitrary) such that $\psi = 0 $ in a neighborhood of $\Gamma$. We have, in $\mR^3$,  
\begin{equation*}
\nabla \times (\psi E_n) = \psi \nabla \times E_n + \nabla \psi \times E_n \quad \mbox{ and }  \quad \dive [ \eps_0  (\psi E_n)] = \psi \dive (\eps_0 E_n) + \nabla \psi \cdot  \eps_0 E_n.  
\end{equation*}
Using \eqref{thm2-def-beta} and \eqref{thm2-p1},  and applying \cite[Lemma 1]{Ng-Superlensing-Maxwell}, one may assume that 
$$
(\psi E_n) \mbox{ converges  in } [L^2(\mR^3)]^3, 
$$
which yields,  since $\psi$ is arbitrary, 
\begin{equation}\label{thm2-p1-*}
(E_n) \mbox{ converges  in } [L^2_{\loc}(\mR^3 \setminus \Gamma)]^3. 
\end{equation}
Similarly, one may also assume that 
\begin{equation}\label{thm2-p2-*}
(H_n) \mbox{ converges  in } [L^2_{\loc}(\mR^3 \setminus \Gamma)]^3.  
\end{equation}
Moreover,   from  \eqref{thm2-p2}, \eqref{thm2-p1-*}, \eqref{thm2-p2-*}   and the fact $\beta > \max\{\alpha_1, \alpha_2\}$,  by Lemma~\ref{lem-Inter}, one may assume that
\begin{equation}\label{thm2-contradiction-2}
\big( (E_n, H_n) \big) \mbox{ converges in } [L^2(d_\Gamma^{\beta}, B_{R_0}) ]^6
\end{equation}
and, by  \eqref{thm2-p2-1}, 
\begin{equation}\label{thm2-contradiction-2-1}
\big( (\dEn, \dHn) \big) \mbox{ converges weakly in  $[L^2(D_{-\tau})]^6$}. 
\end{equation}
Let $(E, H)$ be the limit of $\big( (E_n, H_n) \big) $ in $[L^2_{\loc}(\mR^3 \setminus \Gamma)]^6$. From \eqref{thm2-p2-1}, we derive that  $(\dE, \dH) \in [L^2(D_{-\tau})]^2$. From the equations of $(E, H)$, it follows that $(\dE, \dH) \in [H(\curl, D_{-\tau})]^2$. One also has 
\begin{equation}\label{thm2-TC}
\dE \times \nu = \dH \times \nu = 0 \mbox{ on } \Gamma. 
\end{equation}

We have, for $R> R_0$,  
\begin{equation*}
\left| \Re \int_{\partial B_R} H \times \nu \cdot \bar E \right| =  \lim_{n \to + \infty} \left|  \Re \int_{\partial B_R} H_n \times \nu \cdot \bar E_n \right| \mathop{\le}^{\eqref{thm2-p0-111}} \limsup_{n \to + \infty}   \left| \int_{B_R} J_n  \bar E_n \right| \mathop{=}^{ \eqref{thm2-contradiction-1}} 0. 
\end{equation*}
Since $(E, H)$ satisfies the radiating condition, it follows that 
\begin{equation*} 
E = H = 0 \mbox{ in the unbounded connected component of $\mR^3 \setminus \bar D$} , 
\end{equation*}
which in turn implies, by \eqref{thm2-TC} and  the unique continuation principle,  that 
\begin{equation}\label{thm2-contradiction-3}
E = H = 0 \mbox{ in } \mR^3. 
\end{equation}
Combining \eqref{thm2-contradiction-1}, \eqref{thm2-contradiction-2},  \eqref{thm2-contradiction-2-1}, and \eqref{thm2-contradiction-3} yields a contradiction.  Hence \eqref{thm2-p1} is proved. 
Applying Lemma~\ref{lem-M2} again,  we derive \eqref{T2-1} from \eqref{thm2-p1}.

By the same method, one also obtains the following fact: for any $(\delta_n) \to 0$, up to a subsequence, $\big((  E_{\delta_n}, H_{\delta_n}) \big)$ converges in $[L^2_{\loc} (\mR^3 \setminus \Gamma)]^6$, and the limit  is a radiating solution of the corresponding system.  

\medskip 
We next consider the case $\alpha_1 = \alpha_2 = 0$. We first prove by contradiction that 
\begin{equation*}
\| (E_\delta, H_\delta) \|_{L^2(B_{R_0})}  \le C \| J \|_{L^2(\mR^3)}. 
\end{equation*}
Assume, for some $\delta_n \to 0$ and $J_n \in L^2(\mR^3)$ with $\supp J_n \subset B_{R_0}$ that 
\begin{equation}\label{thm2-contradiction-1-***}
\lim_{n \to + \infty} \| J_n\|_{L^2(\mR^3)} = 0 
\quad \mbox{ and } \quad   \| (E_n, H_n) \|_{L^2(B_{R_0})}
 = 1, 
\end{equation}
where $(E_n, H_n)$ is the solution corresponding to $\delta_n$ and $J_n$. 
By a change of variables for the Maxwell equations, see e.g.  \cite[Lemma 7]{Ng-Superlensing-Maxwell}, we have 
\begin{equation*}
\left\{ \begin{array}{cl}
\nabla \times E_n = i \omega \mu H_n  & \mbox{ in } D_{-\tau}, \\[6pt]
\nabla \times H_n = - i \omega \eps E_n  + J_n & \mbox{ in } D_{-\tau}, 
\end{array}\right. \quad \quad 
\left\{ \begin{array}{cl}
\nabla \times \hE_n = i \omega  \hmu \hH_n + \hJ_{e, n} & \mbox{ in } D_{-\tau}, \\[6pt]
\nabla \times \hH_n = - i \omega  \heps \hE_n + \hJ_{m, n} & \mbox{ in } D_{-\tau}, 
\end{array}\right. 
\end{equation*}
and 
\begin{equation*}
E_n \times \nu = \hE_n \times \nu  \quad \mbox{ and } \quad H_n \times \nu = \hH_n \times \nu \quad  \mbox{ on } \Gamma. 
\end{equation*}
Here 
$$
\hJ_{e, n} =  - \delta \omega   {\mathcal F} _*I  \, \hH_n \quad \mbox{ and } \quad \hJ_{m, n} =  \delta \omega {\mathcal F}_*I \, \hE_n  +  {\mathcal F}_*J_n  \quad  \mbox{ in } D_{- \tau}.  
$$
Note that 
$$
\int_{\Omega} |J_n -  \hJ_{m, n}| |\hE_n| + |\hJ_{e, n} | |\hH_n| \le \left( \int_{\Omega} |(J_n, \hJ_{m, n}, \hJ_{e, n})|^2 \right)^{1/2}  \left( \int_{\Omega} |(\hE_n, \hH_n)|^2 \right)^{1/2}. 
$$
As in \eqref{thm2-p2}, we then have 
\begin{equation*}
\| E_n \|_{L^2(B_{R_0})} + \| H_n \|_{L^2(B_{R_0})}  \le C. 
\end{equation*}
As in the proof of Theorem~\ref{thm1}, one can prove that, up to a subsequence, 
\begin{equation}\label{thm2-contradiction-1-******}
(E_n, H_n) \mbox{ converges in } [L^2_{\loc} (\mR^3)]^{12}.
\end{equation}
Moreover, the limit $(E_0, H_0)$ is in $[H_{\loc}(\curl, \mR^3)]^2$ and is a radiating solution of the equations 
\begin{equation*}
\left\{\begin{array}{cl}
\nabla \times E_0 = i \omega \mu_0 H &  \mbox{ in } \mR^3, \\[6pt]
\nabla \times H_0 = - i \omega \eps_0 E&  \mbox{ in } \mR^3.  
\end{array} \right.
\end{equation*}
By Lemma~\ref{lem-unique-1}, we get
$$
E_0 = H_0 = 0 \mbox{ in } \mR^3. 
$$
This contradicts \eqref{thm2-contradiction-1-***} and \eqref{thm2-contradiction-1-******}. The uniqueness of $(E_0, H_0)$ is a consequence of Lemma~\ref{lem-unique-1} and the strong convergence of $(E_\delta, H_\delta)$ to $(E_0, H_0)$ in $[L^2_{\loc} (\mR^3)]^{12}$ can be derived as in the proof of Theorem~\ref{thm1} and are omitted. \qed

\subsection{Some applications of Theorem~\ref{thm2}} \label{sect-application}

In this section, we present two applications of Theorem~\ref{thm2}.  The first is a consequence of Theorem~\ref{thm2} with  $\alpha_1 = \alpha_2 = 0$.  We have

\begin{corollary}\label{corisotropic-2}
Let $0 < \delta < 1$,  $J \in [L^2(\mR^3)]^3$ with $\supp J \subset B_{R_0}$, and  
let $(E_\delta, H_\delta) \in [H_{\loc} (\curl, \mR^3)]^2$ be the unique radiating solution of \eqref{Main-eq-delta}. Assume that $D$ is of class $C^2$ and,   for each connected component $O$ of $D_{-\tau} \cup \Gamma \cup D_{\tau}$ with $\tau > 0$ small, the following four conditions hold 
\begin{equation*}
\mbox{ either } \ep |_{D_{-\tau} \cap O} \mbox { or } \en |_{D_{\tau} \cap O} \mbox{  is isotropic},  
\end{equation*}
\begin{equation*}
\mbox{ either } \mup  |_{D_{-\tau} \cap O}  \mbox { or } \mun |_{D_{\tau} \cap O} \mbox {  is isotropic},  
\end{equation*}
\begin{equation*}
\ep \big(x_\Gamma - t \nu(x_\Gamma) \big)  \ge - \en \big(x_\Gamma + t \nu(x_\Gamma) \big)  +  c I  \quad \mbox{ or } \quad 
- \en \big(x_\Gamma - t \nu(x_\Gamma) \big) \geq   \ep \big(x_\Gamma + t \nu(x_\Gamma) \big) + c I 
\end{equation*}
and
\begin{equation*}
\mup \big(x_\Gamma - t \nu(x_\Gamma) \big) \ge - \mun \big(x_\Gamma + t \nu(x_\Gamma) \big) +  c I  \quad \mbox{ or } \quad 
- \mun \big(x_\Gamma - t \nu(x_\Gamma) \big) \ge   \mup \big(x_\Gamma + t \nu(x_\Gamma) \big) +  c I, 
\end{equation*}
 for every $x_{\Gamma} \in \Gamma \cap O$ and for every $t \in (0, \tau)$,  for some $c>0.$  Then, for all $R> 0$, 
\begin{equation*}
\int_{B_R} |(E_\delta, H_\delta)|^2  \le C_R \| J\|_{L^2(\mR^3)}^2, 
\end{equation*}
for some positive constant $C_{R}$ independent of $\delta$ and $J$.  
Moreover,  $(E_\delta, H_\delta)$ converges to $(E_0, H_0)$ strongly in $[L^2_{\loc}(\mR^3)]^6$, as $\delta \to 0$, where  $(E_0, H_0) \in [H_{\loc}(\curl, \mR^3 )]^2$ is the  unique radiating  solution of \eqref{Main-eq-delta} with $\delta = 0$.  As a consequence,  
\begin{equation*} 
\int_{B_R} |(E_0, H_0)|^2    \le C_{R} \| J\|_{L^2(\mR^3)}^2. 
\end{equation*}
\end{corollary}

\begin{remark} \rm
It is worth comparing Corollary \ref{corisotropic-2} with Corollary \ref{corADN}. First,  in Corollary \ref{corisotropic-2},  one does not require $\ep, \en, \mup, \mun$  to be $C^{1}$ near $\Gamma$.  Second,   in Corollary \ref{corADN}, one  does not require any isotropy conditions on $\ep, \en, \mup, \mun$.  
\end{remark}

\begin{proof} For $\tau > 0$ sufficiently small (the smallness depends only on $D$), define  ${\mathcal F}: D_{\tau} \cup \Gamma \cup  D_{-\tau} \to D_{\tau} \cup \Gamma \cup  D_{-\tau}$ by
\begin{equation*}
{\mathcal F}(x_\Gamma + t \nu(x_\Gamma)) = x_\Gamma - t \nu(x_\Gamma) \quad \forall \, x_\Gamma \in \Gamma, \, t \in (-\tau, \tau) 
\end{equation*}
and set,  in  $D_{-\tau}$,  
\begin{equation*}
(\heps, \hmu) =  ({\mathcal F}_*\en, {\mathcal F}_*\mun). 
\end{equation*}
For the simplicity of the presentation, we assume that $D_{-\tau}$ is connected. We will only consider the case
\begin{equation*}
\ep |_{D_{-\tau} }  \mbox{  is isotropic},  
\end{equation*}
\begin{equation*}
\mup  |_{D_{-\tau} }  \mbox {  is isotropic},  
\end{equation*}
\begin{equation*}
\ep \big(x_\Gamma - t \nu(x_\Gamma) \big)  \ge - \en \big(x_\Gamma + t \nu(x_\Gamma) \big)  +  c I,   
\end{equation*}
and
\begin{equation*}
\mup \big(x_\Gamma - t \nu(x_\Gamma) \big) \ge - \mun \big(x_\Gamma + t \nu(x_\Gamma) \big) +  c I 
\end{equation*}
 for every $x_{\Gamma} \in \Gamma \cap O$ and for every $0 < t < \tau$, for some $c>0$. The other cases can be dealt similarly. 
 
 For  $x' \in \bar D_{-\tau}$, set $x = {\mathcal F}^{-1}(x')$ and ${\mathcal J}(x) = \det \nabla {\mathcal F}(x)$. We have, for $x \in \Gamma$, 
$$ 
\nabla {\mathcal F}^{-1} (x)  \nabla {\mathcal F}^{-T} (x) =  I \mbox{ and } \quad {\mathcal J} (x) = - 1. 
$$ 
Since $\ep$ is isotropic and $\ep \big(x_\Gamma - t \nu(x_\Gamma) \big)  \ge - \en \big(x_\Gamma + t \nu(x_\Gamma) \big)  +  c I$ in $D_{-\tau}$, it follows from the definition of $\heps$ that, for sufficiently small $\tau$, 
\begin{equation}\label{corisotropic-2-1}
\eps(x') \ge \heps(x') + c I/2 \mbox{ for } x' \in D_{-\tau}. 
\end{equation}
Similarly, we have, for sufficiently small $\tau$,
\begin{equation}\label{corisotropic-2-2}
\mu(x') \ge \hmu(x') + c I/2 \mbox{ for } x' \in D_{-\tau}. 
\end{equation}
Take $\tau$ sufficiently small such that \eqref{corisotropic-2-1} and \eqref{corisotropic-2-2} hold.  Applying Theorem~\ref{thm2}, we obtain the conclusion. 
\end{proof}

Here is a direct consequence of Corollary \ref{corisotropic-2}. Let $\Omega_1$ and $\Omega_2$ be open subsets of $D$ such that $\Omega_1 \cap \Omega_2 = \emptyset$,  $\bar D = \bar \Omega_1 \cup \bar \Omega_2$, $\partial \Omega_1 \cap \partial D  \neq \emptyset$, and $\partial \Omega_2 \cap \partial D  \neq \emptyset$. Assume that $\eps^- = - \gamma_1 I $ in $\Omega_1$ and $ = - \gamma_2 I $ in $\Omega_2$, and $\eps^+ = I$ in $\mR^3 \setminus D$ for some constants $\gamma_1, \gamma_2 >0$ with $\min\{ \gamma_1, \gamma_2\} > 1$ or $\max\{ \gamma_1, \gamma_2\} < 1$. For $J \in L^2(\mR^3)$ with $\supp J \subset B_{R_0}$, let $(E_\delta, H_\delta) \in [H_{\loc} (\curl, \mR^3)]^2$ be the unique radiating solution of \eqref{Main-eq-delta} for $0  < \delta < 1$. We have $(E_\delta, H_\delta)$ is bounded in $[L^2_{\loc}(\mR^3)]^6$. Moreover, $(E_\delta, H_\delta)$ converges to $(E_0, H_0)$ strongly in $[L^2_{\loc}(\mR^3)]^6$, the unique radiating solution of \eqref{Main-eq-delta} with $\delta = 0$. Note that this setting  is out of the scope of Corollary~\ref{corADN} for  $\gamma_1 \neq \gamma_2$.

\medskip 

Here is another consequence of Theorem~\ref{thm2} for which $\alpha_1  = \alpha_2 = 1$. We first introduce
\begin{definition}  \rm A connected component of  $\Gamma$ is called strictly convex if it is  the boundary of a strictly convex set.
\end{definition}

We have

\begin{corollary}\label{corisotropic-3}
Let $0 < \delta < 1$, $\tau > 0$, $J \in [L^2(\mR^3)]^3$ with $\supp J \subset  B_{R_0} \setminus (D_\tau \cup D_{-\tau})$, and  
let $(E_\delta, H_\delta) \in [H_{\loc} (\curl, \mR^3)]^2$ be the unique radiating solution of \eqref{Main-eq-delta}. Let $\eeps, \mmu$ be positive symmetric matrix-valued functions defined in $D_{-\tau} \cup \Gamma \cup D_{\tau}$ such that they are {\rm constant} and {\rm isotropic} on each connected component of their domain of definition. 
Assume that 
$$
\mbox{$D$ is of class $C^3$,  each connected component of $\Gamma$  is  strictly convex}, 
$$
\begin{equation*}
(\ep, \mup) = (\eeps, \mmu) \mbox{ in } D_{-\tau}, \quad \mbox{ and } \quad (\en, \mun)  =  -  (\eeps, \mmu) \mbox{ in } D_{\tau}.
\end{equation*}
Then, for $R> 0$ and open  $ V \supset \Gamma$,  
\begin{equation}\label{T1-1-coucou}
\| (E_\delta, H_\delta) \|_{L^2(B_R \setminus V)} \le C_{R, V} \| J\|_{L^2(\mR^3)} \quad \forall \, R > 0, 
\end{equation}
for some positive constant $C_{R, V}$ independent of $\delta$ and $J$.  
Moreover,  for a sequence $(\delta_n) \to 0$, up to a subsequence,  $(E_{\delta_n}, H_{\delta_n})$ converges to $(E_0, H_0)$ strongly in $L^2_{\loc}(\mR^3 \setminus \Gamma)$ as $n \to + \infty$, where  $(E_0, H_0) \in [H_{\loc}(\curl, \mR^3 \setminus \Gamma)]^2$ is a radiating solution of \eqref{Main-eq-delta} with $\delta = 0$ and \eqref{T1-1-coucou} holds with $\delta = 0$. 
\end{corollary}

\begin{proof}  For simple presentation, we assume that $D$ is convex. Let ${\mathcal F}$ be defined as follows:  
\begin{equation*}
x_\Gamma - t \nu (x_\Gamma) \mapsto x_\Gamma + t \big[1 + t c(x_\Gamma) \big] \nu(x _\Gamma),   
\end{equation*}
for $x_\Gamma \in \Gamma$ and $0 < t < \tau$ (small). Here  $c(x_\Gamma) = \beta \mbox{trace} \Pi(x_\Gamma)$ where $\Pi(x_\Gamma)$ is the second fundamental form of $\Gamma$ at $x_\Gamma$ and $-1 < \beta < 0$. By taking $\beta$ sufficiently close to $-1$, one can prove that, see \cite[Proof of Corollary 3]{Ng-WP}, 
$$
{\mathcal F}_*(-I) - I \ge \gamma d_{\Gamma} I \mbox{ in } D_{-\tau}.  
$$
for some positive constant $\gamma$ (note that the definition of ${\mathcal F}_*$ in this paper is different from the one in \cite{Ng-WP} in which $|\det (\nabla {\mathcal F})|$ is used instead of $\det (\nabla {\mathcal F})$). We are now in the range of the application of Theorem~\ref{thm2} and the conclusion follows. 
\end{proof}

 Here is an immediate application of Corollary~\ref{corisotropic-3}. Let $D$ be a  smooth strictly convex set of class $C^3$ and assume that, for $\delta \ge 0$,  
\begin{equation*}
(\eps_\delta, \mu_\delta)  = \left\{\begin{array}{cl}
I, I & \mbox{ in } \mR^3 \setminus D, \\[6pt]
- I  + i \delta I, - I  + i \delta I  & \mbox{ in } D. 
\end{array} \right. 
\end{equation*}
Let $J \in [L^2(\mR^3)]^3$ with $\supp J \subset B_{R_0} \setminus (D_\tau \cup D_{-\tau})$ and let $(E_\delta, H_\delta)$ be the unique radiating solution of \eqref{Main-eq-delta}. As a consequence of Theorem~\ref{thm2} and \cite[Corollary 3]{Ng-WP},  one has
\begin{equation*}
\| (E_\delta, H_\delta) \|_{L^2(B_R \setminus V)} \le C_{R, V} \| J\|_{L^2(\mR^3)} \quad \forall \, R > 0, 
\end{equation*}
for some positive constant $C_R$ independent of $\delta$ and $J$. Corollary~\ref{corisotropic-3} is a variant of \cite[Corollary 3]{Ng-WP} for the Maxwell equations.

\appendix
\section{Proof of Proposition~\ref{pro-opt}}

We only consider the case  $(\mu, \hmu )$ does not satisfy the complementing  conditions at some point $x_0 \in \partial \Omega$. The other case can be dealt similarly.  Then,  from  \cite{ADNII},  there exist sequences
 $\big( ( u_{n}, \hat{u}_{n} ) \big)\subset H^{2}(\Omega),$ 
 $\big( (f_{n}, \hat f_{n}) \big)  \subset L^{2}(\Omega),$ $ \big( p_{n} \big) \subset H^{\frac{3}{2}}(\partial\Omega)$ and 
 $\big( q_{n} \big)  \subset H^{\frac{1}{2}}(\partial\Omega)$ such that 
 \begin{equation*}
 \begin{array}{c}
  \operatorname*{div} \left( \mu \nabla u_{n} \right) = f_{n}, \quad  \quad   \operatorname*{div} \left( \hmu \nabla \hat{u}_{n} \right) 
  = \hat f_{n} \quad \text{ in } \Omega, \\[6pt]
  u_{n} -  \hat{u}_{n} = p_{n} \quad \text{ on } \partial\Omega, \\[6pt]
   \left( \mu\nabla u_{n} - \hmu \nabla \hat{u}_{n} \right) \cdot \nu = q_{n} \quad \text{ on } \partial\Omega, 
  \end{array}
 \end{equation*}
$ \big( \left\lVert f_{n}\right\rVert_{L^{2}(\Omega)} \big), \,  \big(\lVert \hat f_{n}\rVert_{L^{2}(\Omega)} \big), 
\big(\left\lVert p_{n}\right\rVert_{H^{\frac{3}{2}}(\partial\Omega)} \big), \,  \big(\left\lVert q_{n}\right\rVert_{H^{\frac{1}{2}}(\partial\Omega)} \big)$ and $ 
\big( \left\lVert \left( u_{n}, \hat{u}_{n} \right)\right\rVert_{H^{1}(\Omega)} \big) $ are  bounded and 
$\lim_{n \to + \infty}\left\lVert \left( u_{n}, \hat{u}_{n} \right)\right\rVert_{H^{2}(\Omega)} = + \infty .$ 

Let $\phi_{n} \in H^{2}$ and $\hat{\phi}_{n} \in H^{2}$  be the unique solution of 
\begin{align*}
  \left\{ \begin{aligned}
                \operatorname*{div} \left( \mu \nabla \phi_{n} \right) &= f_{n} &&\text{ in } \Omega, \\
  \phi_{n} &= p_{n} &&\text{ on } \partial\Omega.
               \end{aligned}\right. \qquad \text{ and } \qquad \left\{\begin{aligned}
               \operatorname*{div} \left( \hmu \nabla \hat{\phi}_{n} \right) &= \hat f_{n} &&\text{ in } \Omega, \\
               \hat{\phi}_{n} &= 0 &&\text{ on } \partial\Omega.
               \end{aligned}\right.\end{align*}
By the standard theory of elliptic equations, one has 
 $$ \left\lVert \phi_{n} \right\rVert_{H^{2}} \leq C \left(  \left\lVert f_{n} \right\rVert_{L^{2}} +\left\lVert p_{n} \right\rVert_{H^{\frac{3}{2}}(\partial\Omega)} \right) \quad \text{ and } \quad \left\lVert \hat{\phi}_{n} \right\rVert_{H^{2}} \leq C \left\lVert \hat f_{n} \right\rVert_{L^{2}}, 
 $$
 for some positive constant $C$ independent of $n$.  Set
$$
H_{n} = \nabla u_{n} - \nabla \phi_{n} \quad \mbox{ and } \quad 
\hH_{n} = \nabla \hat{u}_{n} - \nabla \hat{\phi}_{n}, \mbox{ in } \Omega.
$$ 
We have, in $\Omega$, 
$$  \operatorname*{curl} H_{n} = 0 =  \operatorname*{curl} \hH_{n} \quad \text{ and } \quad  
\operatorname*{div} \left( \mu H_{n} \right) = 0 =  \operatorname*{div} \big( \hmu  \hH_{n} \big).
$$
Since $u_{n} -  \hat{u}_{n} = p_{n} = \phi_{n}$ on $\partial\Omega,$ and $\hat \varphi_n = 0$ on $\partial \Omega$,  we deduce  that 
$$ \nu \times \big( H_{n} - \hH_{n} \big) = 0 \qquad \text{ on } \partial\Omega.$$

Set 
$$
E_n = \hat E_n = 0 \mbox{ in } \Omega 
$$
and, in $\Omega$,  
$$ 
J_{e,n} = -i\omega \mu H_{n}, \quad \hJ_{e,n} = -i\omega \hmu \hH_{n}, \quad 
J_{m,n} =0 , \quad \mbox{ and } \quad \hJ_{m,n} = 0.
$$ 
One can easily check that $(E_n, \hE_n, H_n, \hH_n)$ and $(J_{e, n}, \hJ_{e, n}, J_{m, n}, \hJ_{e, n})$ satisfies all the required properties. \qed

\providecommand{\bysame}{\leavevmode\hbox to3em{\hrulefill}\thinspace}
\providecommand{\MR}{\relax\ifhmode\unskip\space\fi MR }
\providecommand{\MRhref}[2]{%
  \href{http://www.ams.org/mathscinet-getitem?mr=#1}{#2}
}
\providecommand{\href}[2]{#2}

\end{document}